\pdfoutput=1

\documentclass[11pt]{amsart}
\usepackage[centertags]{amsmath}
\usepackage{amsfonts}
\usepackage{amssymb}
\usepackage{amsthm}
\usepackage{epsfig}
\usepackage{fullpage }
\usepackage{MnSymbol}
\usepackage{subfigure}
\usepackage{color}

\newtheorem{theorem}{Theorem}[section]
\newtheorem{proposition}[theorem]{Proposition}
\newtheorem{lemma}[theorem]{Lemma}

\theoremstyle{definition}
\newtheorem{definition}[theorem]{Definition}

\theoremstyle{remark}
\newtheorem{remark}[theorem]{Remark}
\newtheorem{example}[theorem]{Example}

\begin{document}

\newcommand{\G}{\mathbb{G}}
\newcommand{\T}{T}
\newcommand{\bs}{\backslash}

\newcommand{\C}{C}
\newcommand{\D}{\mathcal{D}}
\newcommand{\fD}{\mathfrak{D}}
\newcommand{\Int}{\mathrm{Int}}

\newcommand{\bi}{biseparation }
\newcommand{\jbi}{join-biseparation }
\newcommand{\jbit}{join-biseparation}

\newcommand{\jbis}{join-biseparations }
\newcommand{\bis}{biseparations }
\newcommand{\bit}{biseparation}
\newcommand{\bist}{biseparations}

\newcommand{\pbi}{plane-biseparation }
\newcommand{\pbis}{plane-biseparations }
\newcommand{\pbit}{plane-biseparation}
\newcommand{\pbist}{plane-biseparations}

\newcommand{\pjbi}{plane-join-biseparation }
\newcommand{\pjbis}{plane-join-biseparations }
\newcommand{\pjbit}{plane-join-biseparation}
\newcommand{\pjbist}{plane-join-biseparations}

\newcommand{\RPt}{$\mathbb{R}\mathrm{P}^2$}

\newcommand{\RP}{$\mathbb{R}\mathrm{P}^2$ }
\newcommand{\rp}{\mathbb{R}\mathrm{P}^2 }

\newcommand{\rpbi}{$\mathbb{R}\mathrm{P}^2$-biseparation }
\newcommand{\rpbit}{$\mathbb{R}\mathrm{P}^2$-biseparation}
\newcommand{\rpbis}{$\mathbb{R}\mathrm{P}^2$-biseparations }
\newcommand{\rpbist}{$\mathbb{R}\mathrm{P}^2$-biseparations}

\newcommand{\rpjbi}{$\mathbb{R}\mathrm{P}^2$-join-biseparation }
\newcommand{\rpjbis}{$\mathbb{R}\mathrm{P}^2$-join-biseparations }
\newcommand{\rpjbit}{$\mathbb{R}\mathrm{P}^2$-join-biseparation}
\newcommand{\rpjbist}{$\mathbb{R}\mathrm{P}^2$-join-biseparations}

\newcommand{\ga}{\gamma}

%\referee
%\linenumbers

\title[]{Separability and the genus of a partial dual}
%\dedicatory{Working copy. Do not distribute!}

\author[I.~Moffatt]{Iain Moffatt$^\dagger$}

\begin{abstract}
Partial duality  generalizes the fundamental concept of the geometric dual of an embedded graph. A partial dual is obtained by forming the geometric dual with respect to only a subset of edges. While geometric duality preserves the genus of an embedded graph, partial duality does not. 
 Here we are interested in the problem of determining which edge sets of an embedded graph give rise to a partial dual of a given genus. 
 This problem turns out to be intimately connected to the separability of the embedded graph. We determine how separability is related to the genus of a partial dual. We use this to characterize partial duals of graphs embedded in the plane, and in the real projective plane, in terms of  a particular type of separation of an embedded graph. 
 These characterizations are then used to determine a local move relating all partially dual graphs in the plane and in the real projective plane.

\end{abstract}

\keywords{ $1$-sum, dual, embedded graph, join,   map amalgamation, partial dual, plane graph,  real projective plane,  ribbon graph, separability.}

\subjclass[2010]{Primary: 05C10. Secondary:  05C75.}

\thanks{
${\hspace{-1ex}}^\dagger$
	Department of Mathematics,
Royal Holloway,
University of London,
Egham,
Surrey,
TW20 0EX,
United Kingdom;\\
${\hspace{.35cm}}$ \texttt{iain.moffatt@rhul.ac.uk}}

\date{\today}

\maketitle
%\tableofcontents

\section{Introduction}
The geometric dual $G^*$ of an embedded graph $G$ is a fundamental construction in graph theory, and is one that appears throughout mathematics. If $G$ is cellularly embedded in a surface $\Sigma$, then its dual $G^*$ is also cellularly embedded in $\Sigma$.  Thus geometric duality does not change the genus of an embedded graph.

Recently, in \cite{Ch1}, S.~Chmutov  introduced the concept of  partial duality, which is a far-reaching extension of geometric duality.   Roughly speaking,  the partial dual $G^A$ of an  embedded graph $G$ is obtained by forming the geometric dual with respect  to only a subset $A$ of edges  of the graph. (A formal definition is given Subsection~\ref{ss.pd}). The partial dual formed with respect to the entire edge set of $G$ is  the  geometric dual. 
Partial duality arose out of knot theory and has found a number of applications in graph theory, topology, and physics 
(see, for example, \cite{Ch1,EMM,EMM2,HM11,HMV11,KRVT09,Mo2,Mo3,Mo4,Mo5,VT,VT10}). 

Possibly the most immediate difference between geometric duality and its generalization, partial duality, is that while  geometric duality always preserves the  genus of an embedded graph, partial duality does not. For example, if $G$ is a plane graph, then $G^*$ is also a 
plane graph, but a partial dual $G^A$ of $G$ need not be plane. It is this property of partial duality that we are interested in here. We address the following problem:
\begin{itemize}
\item For an embedded graph $G$, determine the subsets $A\subseteq E(G)$ that give rise to a partial dual $G^A$ of a given genus.
\end{itemize}
This problem turns out to be intimately connected with the separability of the embedded graph $G$. 

Recall that a graph is said to be {\em separable} if it admits a decomposition into two non-empty, connected subgraphs that have only a vertex in common. Such a pair of subgraphs is called a {\em separation} of $G$, and the vertex where they meet is called a {\em separating vertex}. If $A\subseteq E(G)$, then we say that $A$ {\em defines} a separation if the induced subgraph  $G|_A$ and its complementary induced subgraph $G|_{E(G)\bs A}$ define a separation of $G$, {\em i.e.},   $G|_A$ and $G|_{E(G)\bs A}$ are non-empty, connected and intersect in a separating vertex. 

In general the induced subgraphs $G|_A$ and $G|_{E(G)\bs A}$ will not be connected. Here we extend the concept of a separation to cope with this situation, introducing the idea of a {\em \bit}.
(These are  defined formally in Subsection~\ref{ss.bi}.) Loosely speaking, a subset of edges $A$ defines a \bi of $G$ if it induces a decomposition of $G$ into two (not necessarily connected) subgraphs $G|_A$ and $G|_{E(G)\bs A}$ with the property that any pair of components of $G|_A$ and  $G|_{E(G)\bs A}$ have at most one vertex in common, and such a common vertex is a separating vertex of $G$. (See Definition~\ref{d1} and Example~\ref{e.bis}.) 

Equipped with this idea, we connect the genus of a partial dual and the separability of an embedded graph by showing that $A$ defines a \bi of an embedded graph $G$ if and only if the genus of the partial dual $G^A$ is determined (in a specific and simple way) by the genera of the induced subgraphs (see Theorem~\ref{t1}).
We  then apply this result to show that the classes of embedded graphs that are  partial duals of graphs embedded in the plane, or real projective plane \RPt, can be  completely characterized in terms the existence of  \bist. (See Theorem~\ref{t2}.)  These characterizations provide a common framework for working with partial duals of graphs in the plane and in \RPt. We  use this framework to characterize partially dual plane graphs and partially dual \RP graphs (Theorem~\ref{t3}), finding a local move on embedded graphs that relates all partially dual plane graphs and all partially dual \RP graphs (Theorem~\ref{t4}). In addition, we discuss why \bis fail to characterize partial duals of higher genus embedded graphs.

The relationships between separability and the genus of a partial dual presented here have their origins in knot theory. In \cite{Mo5}, \pbis were introduced  in order to characterize the class of ribbon graphs that present link diagrams (in the sense of \cite{Detal}), and to relate link diagrams that are represented by the same set of ribbon graphs.  Although the majority of the results presented here are original, the special cases in Sections~\ref{s4} and \ref{s5} that deal with  plane graphs  are from \cite{Mo5}. As the general results from Section~\ref{s3} connecting \bis and partial duals  provide a unified framework for working with partial duals of low genus graphs,  we include these  results  in order to present a complete picture.
 Full proofs for   known results are only given if the proof is new (as in the proof of Theorem~\ref{t2}).  Otherwise, it is indicated how to adapt the proof of the \RP case  to obtain the plane case, or, in a few cases, a reference to \cite{Mo5} is given.
 
This paper is structured as follows. Section~\ref{s2}  discusses embedded graphs and the various constructions we use in this paper. Section~\ref{s3} introduces \bist, determines their connections with the genus of a partial dual,   and  the set of \bis that an embedded graph admits is studied. In Section~\ref{s4}, partial duals of plane graphs and \RP graphs are characterized in terms of \bist, and the types of \bis that such embedded graphs admit are discussed.  Section~\ref{s5} is concerned with partially dual plane graphs and partially dual \RP graphs. Such graphs are characterized in terms of \bist, and a local move connecting them is given.

In this paper we mostly restrict our results to connected graphs, but the results extend easily to non-connected graphs.

\section{Ribbon graphs and partial duals}\label{s2}
This section contains a description of   basic objects and constructions (including ribbon graphs, partial duals and $n$-sums) that we use in this paper.

\subsection{Ribbon graph and arrow presentations}
In this subsection we review cellularly embedded graphs, ribbon graphs and their representations.

%\begin{figure}
%\begin{tabular}{ccccc}
%\includegraphics[width=40mm]{embgraph} &\raisebox{9mm}{ = }& \includegraphics[width=40mm]{embgraph2} &\raisebox{9mm}{ = }& \includegraphics[width=40mm]{arrpresexamp3} \\
%(i)&&(ii)& &(iii)  \\
%\raisebox{9mm}{ = }\includegraphics[width=40mm]{rgex1} &\raisebox{9mm}{ = }& \includegraphics[width=40mm]{rgex3} &\raisebox{9mm}{ = }& \includegraphics[width=40mm]{rgex2} \\
%(iv)&&(v)& &(vi)  
%\end{tabular}
%\caption{Realizations of an \RP graph.}
%\label{fig.rgex}
%\end{figure}

\begin{figure}
\centering
\subfigure[]{
\includegraphics[scale=.55]{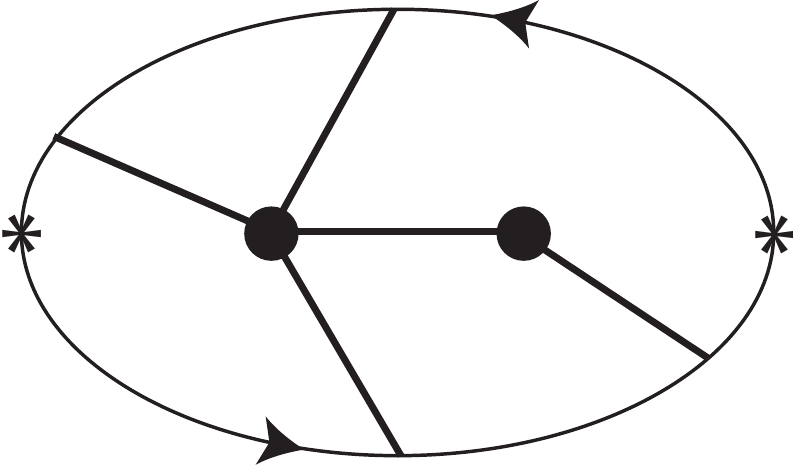}
\label{fig.rgexi}
}
\raisebox{10mm}{\;=\;}
\subfigure[]{
\includegraphics[scale=.55]{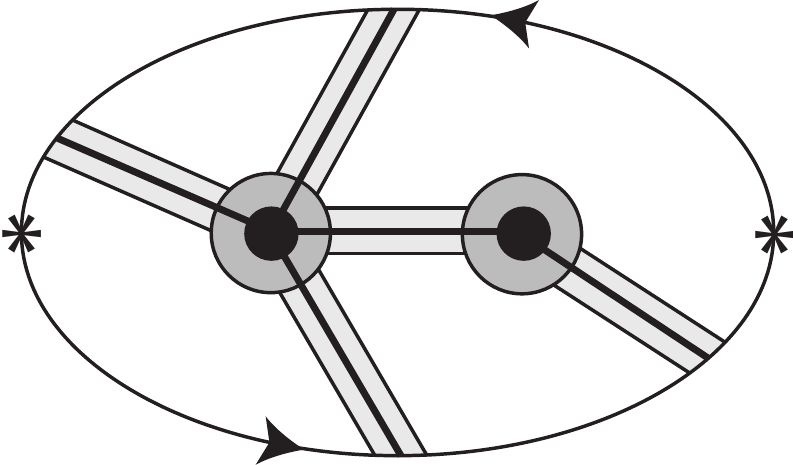}
\label{fig.rgexii}
}
\raisebox{10mm}{\;=\;}
\subfigure[]{
\includegraphics[scale=.4]{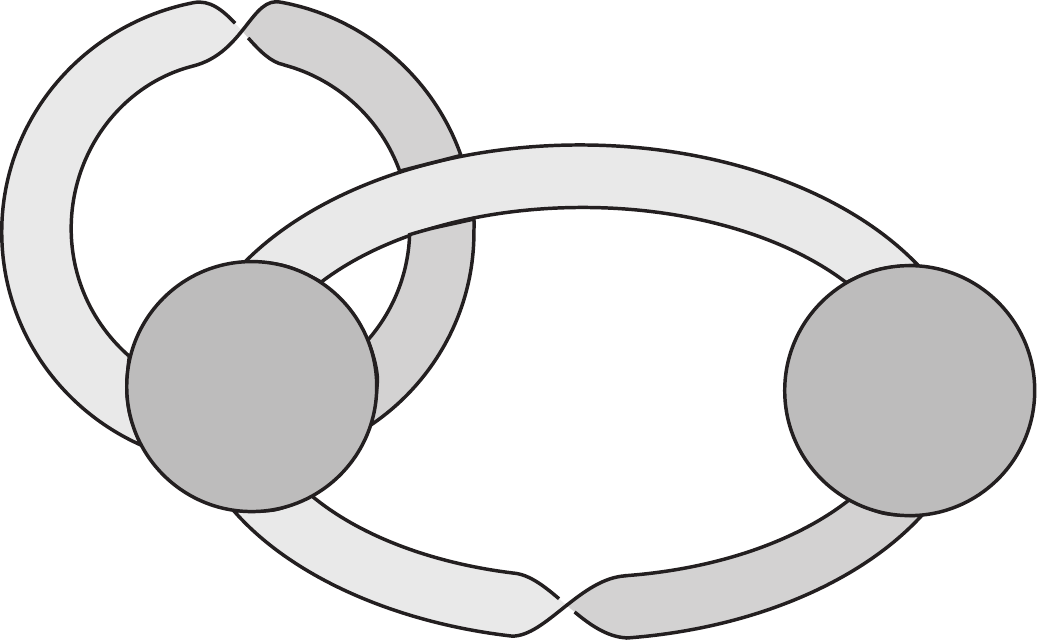}
\label{fig.rgexiii}
} \\
\raisebox{10mm}{\;=\;}
\subfigure[]{
\includegraphics[scale=.4]{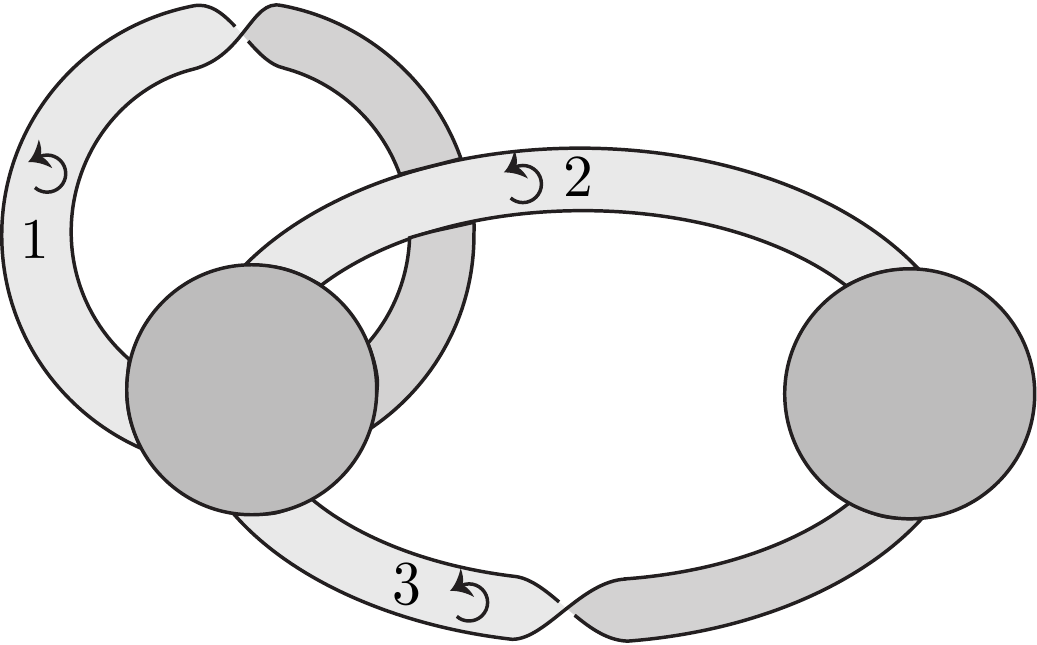}
\label{fig.rgexiv}
}
\raisebox{10mm}{\;=\;}
\subfigure[]{
\includegraphics[scale=.4]{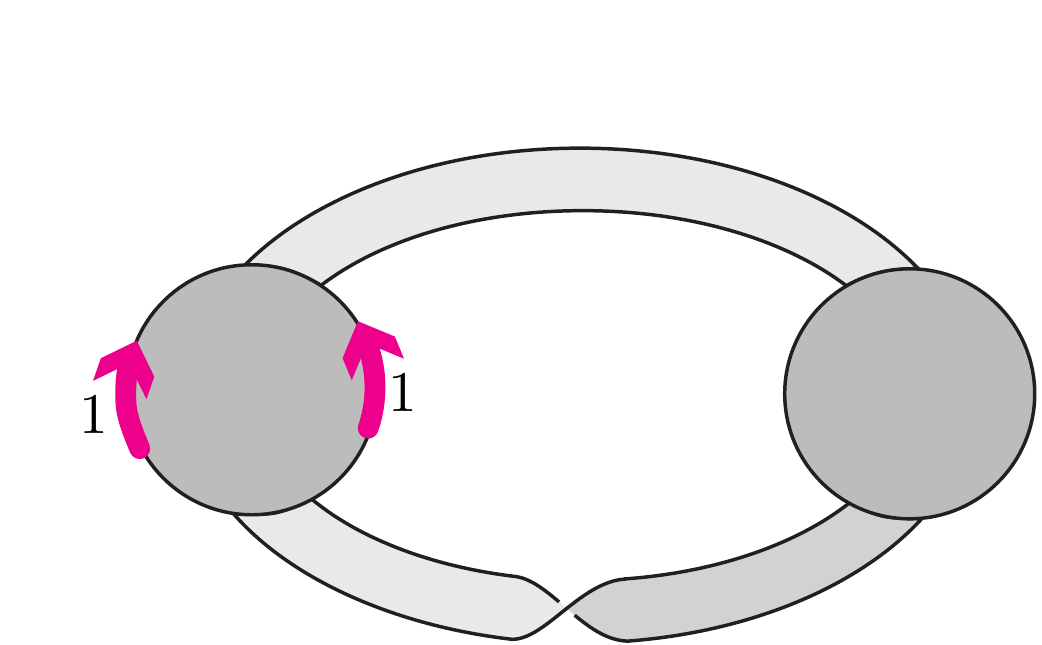}
\label{fig.rgexv}
}
\raisebox{10mm}{\;=\;}
\subfigure[]{
\raisebox{-2mm}{\includegraphics[scale=.5]{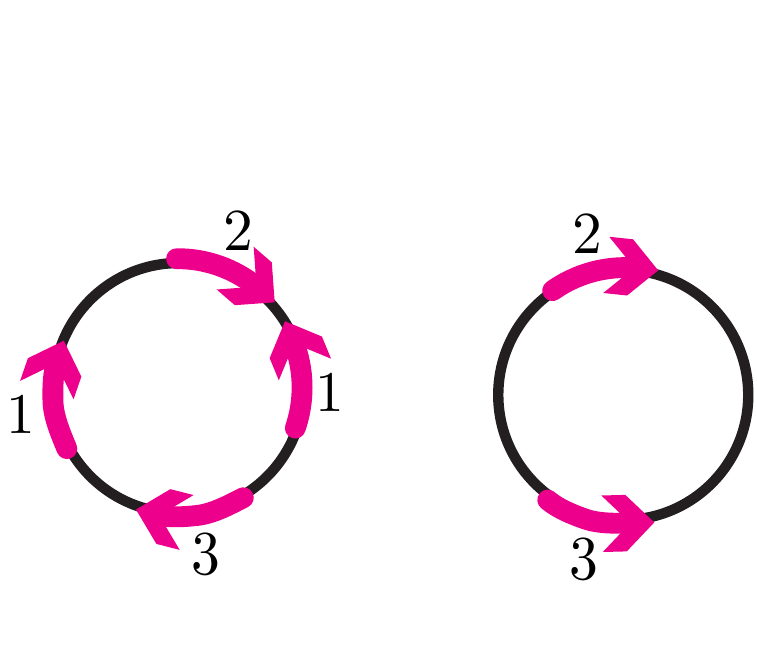}}
\label{fig.rgexvi}
}

\label{fig.rgex}
\caption{Different realizations of the same \RP graph.}
\end{figure}

\subsubsection{Embedded graphs}
An {\em embedded  graph} $G=(V(G),E(G)) \subset \Sigma$ is a graph drawn on a surface $\Sigma$  in such a way that edges only intersect at their ends. The arcwise-connected components of $\Sigma 
\backslash G$ are called the {\em regions} of $G$, and regions homeomorphic to discs are called {\em faces}. If each of the regions of an embedded graph $G$ is a face we say that $G$ is a {\em cellularly embedded graph}. (See Figure~\ref{fig.rgexi} which shows a graph cellularly embedded in the real projective plane.)
Two embedded graphs $G\subset \Sigma$ and  $G'\subset \Sigma'$  are 
 {\em equivalent}   if there is a homeomorphism from  $\Sigma$ 
to $\Sigma'$ that sends $G$ to $G'$. We consider embedded graphs up to equivalence.

We will also consider cellular embeddings of other objects. We say that an object $X\subset \Sigma$ is cellularly embedded if $\Sigma \backslash X$  is a set of discs.

\subsubsection{Ribbon graphs}
In this paper we will primarily work in the language of ribbon graphs. Ribbon graphs describe cellularly embedded graphs, but have the advantage that  deleting edges or vertices of a ribbon graph results in another ribbon graph, whereas deleting an edge of a cellularly embedded graph may not result in a cellularly embedded graph. 

\begin{definition}
A {\em ribbon graph} $G =\left(  V(G),E(G)  \right)$ is a (possibly non-orientable) surface with boundary represented as the union of two  sets of  discs, a set $V (G)$ of {\em vertices}, and a set of {\em edges} $E (G)$ such that: 
\begin{enumerate}%\renewcommand{\labelenumi}{(\roman{enumi})}
\item the vertices and edges intersect in disjoint line segments;
\item each such line segment lies on the boundary of precisely one
vertex and precisely one edge;
\item every edge contains exactly two such line segments.
\end{enumerate}
\end{definition}
 A ribbon graph is shown in Figure~\ref{fig.rgexiii}. The discs are considered up to homeomorphism.

Ribbon graphs are well-known to be (and easily seen to be)  equivalent to cellularly embedded 
graphs.  Intuitively, if $G$ is a cellularly embedded graph, a ribbon 
graph representation results from taking a small neighbourhood  of 
the cellularly embedded graph $G$. On the other hand, if $G$ is a 
ribbon graph,  simply sew discs into each boundary component of the 
ribbon graph  ({\em i.e.}, cap off the punctures) to get a ribbon graph embedded in a surface, and contract the ribbon graph to a graph.  See Figures~\ref{fig.rgexi}-\ref{fig.rgexiii}. 

Two ribbon graphs are {\em equivalent} if they define equivalent cellularly embedded graphs. Ribbon graphs are considered up to equivalence.

At times we will consider cellular embeddings of ribbon graphs. If $G$ is a ribbon graph, then, as  $G$ is topologically a punctured surface, a cellular embedding of $G$ is obtained by capping off its punctures.

%\begin{figure}
%\[\includegraphics[width=40mm]{arrpresexamp3} \quad
% \raisebox{10mm}{\includegraphics[width=18mm]{doublearrow}}\quad
%\raisebox{0mm}{ \includegraphics[width=40mm]{embgraph2}}\quad
% \raisebox{10mm}{\includegraphics[width=18mm]{doublearrow}}\quad
%\raisebox{0mm}{ \includegraphics[width=40mm]{embgraph}}.\]
%\caption{Equivalence of ribbon graphs and 2-cell embeddings.}
%\label{ribbon 2-cell}
%\label{fig.equivRGs}
%\end{figure}

%\begin{figure}
%\begin{tabular}{ccccc}
%\includegraphics[width=40mm]{rgex1} &\raisebox{9mm}{ = }& \includegraphics[width=40mm]{rgex3} &\raisebox{9mm}{ = }& \includegraphics[width=40mm]{rgex2} \\
%(i)&&(ii)& &(iii)  
%\end{tabular}
%\caption{Realizations of a ribbon graph}
%\label{fig.rgex}
%\end{figure}

\subsubsection{Arrow marked ribbon graphs}

We will need to be able to remove edges from a ribbon graph without 
losing any information about their positions. We  do 
this by recording the position of the edges using labelled arrows.

\begin{definition}
An {\em arrow-marked ribbon graph} consists of a ribbon graph  
equipped with a collection of  labelled arrows, called {\em marking 
arrows}, on the boundaries of its vertices. The marking arrows are 
such that no marking arrow meets an edge of the ribbon graph, and    
there are exactly two marking arrows with each label.
\end{definition}

\begin{figure}
\begin{center}
\begin{tabular}{ccccc}
\includegraphics[height=15mm]{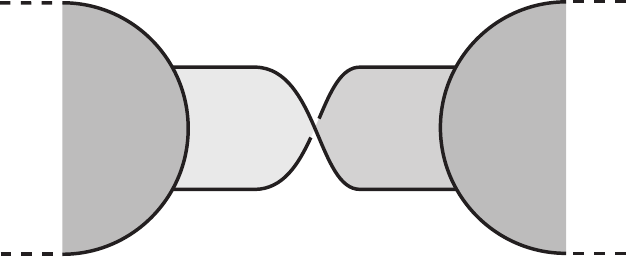}  &
\raisebox{6mm}{\includegraphics[width=11mm]{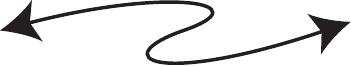}} &
\includegraphics[height=15mm]{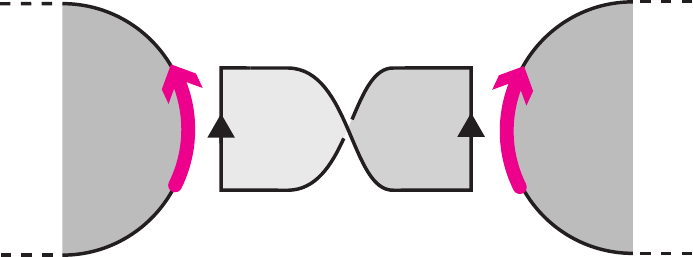}  &
\raisebox{6mm}{\includegraphics[width=11mm]{doublearrow}} &
\includegraphics[height=15mm]{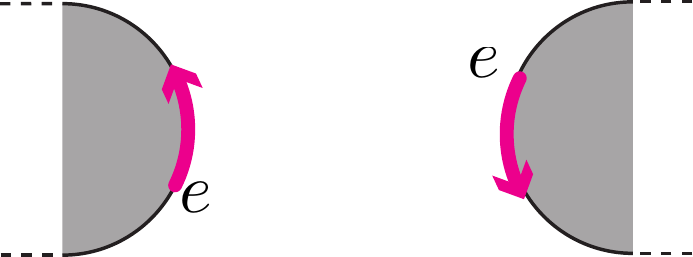} 
\\
$G=H\,\vec{+}\{e\}$  &&&&$G\,\vec{-}\{e\}=H$
\end{tabular}
\end{center}
\caption{Constructing $G\,\vec{-}\{e\}$ and $H\,\vec{+}\{e\}$. }
\label{arrows}
\end{figure}

Let $G$ be a ribbon graph and $A\subseteq E(G)$. Then we let 
$G\,\vec{-} \, A$ denote the arrow-marked ribbon graph obtained, for 
each edge $e\in A$,  as follows: arbitrarily orient the boundary of 
$e$; place an arrow on each of the two arcs where $e$ meets vertices 
of $G$, such that the directions of these arrows follow the 
orientation of the boundary of $e$; label the two arrows with $e$; 
and delete the edge $e$. This process is illustrated, locally at an 
edge, in Figure~\ref{arrows}.

Conversely, given an arrow-marked ribbon graph $H$ with set of labels 
$A$, we can recover a ribbon graph  $H\vec{+}A$ as follows: for each 
label $e\in A$, take a disc and orient its boundary arbitrarily; add 
this disc to the ribbon graph by choosing two non-intersecting arcs 
on the boundary of the disc and the two $e$-labelled marking arrows, and then identifying the arcs with the marking arrows 
according to the orientation of the arrow. The disc that has been 
added forms an edge of a new ribbon graph. Again, this process is 
illustrated in Figure~\ref{arrows}.

See Figures~\ref{fig.rgexiii}-\ref{fig.rgexv} for an example of a ribbon graph and its description as an arrow marked ribbon graph. Further examples can be found in  Figures  \ref{f.ampd1} and \ref{f.ampd2}, and in  Figures \ref{f.ampd3} and \ref{f.ampd4}.

Every arrow-marked ribbon graph corresponds
to a ribbon graph. We  say that two arrow-marked ribbon graphs 
are {\em equivalent} if the ribbon graphs they describe are 
equivalent. We consider arrow-marked ribbon graphs up to equivalence.

We will generally abuse notation and regard the set of labels of an 
arrow-marked ribbon graph as a set of edges. This will allow us to 
view $A$ as an edge set in expressions like $G= \left(G\,\vec{-} \,A\right) \vec{+} A$.

\subsubsection{Arrow presentations}\label{ss.ap}

Every  ribbon graph $G$ has a representation as an arrow-marked ribbon graph $G\vec{-}E(G)$.   In such cases, to describe $G$ it is enough to record only the marked boundary cycles of the vertex set (to recover the vertex set, just place each cycle on the boundary of a disc).  Thus a ribbon graph can be presented as a set of cycles with marking arrows on them. In such a structure, there are exactly two marking arrows with each label. 
Such a structure is called an {\em arrow presentation}.  Formally:
\begin{definition}
An  {\em arrow presentation} of a ribbon graph consists of a set of oriented (topological) circles (called {\em cycles}) that are marked with coloured arrows, called {\em marking arrows}, such that there are exactly two marking arrows of each colour.   
\end{definition} 

A ribbon graph can be recovered from an arrow presentation by regarding the marked cycles as boundaries of discs, giving an arrow-marked ribbon graph.
An example of a ribbon graph and its representation as an arrow presentation is given in 
Figure~\ref{fig.rgexiii} and \ref{fig.rgexvi}. Arrow presentations are {\em equivalent} if they describe the same ribbon graph, and are considered up to equivalence.

\subsubsection{Subgraphs}

A ribbon graph $H =\left(  V(H),E(H)  \right)$ is a {\em ribbon subgraph} of $G =\left(  V(G),E(G)  \right)$ if $H$ can be obtained by deleting vertices and edges of $G$. If $V(H)=V(G)$, then $H$ is a  {\em spanning ribbon subgraph} of $G$.  If $A\subseteq E(G)$, then the ribbon subgraph {\em induced} by $A$, denoted $G|_A$, is the ribbon subgraph of $G$ that consists of the edges in $A$ and their incident vertices. 
We will often regard  $H$ as being embedded in $G$, and will often identify the vertices and edges of $H$ with the corresponding vertices and edges of $G$. 

Throughout the paper we use $A^c:=E(G)\bs A$  to denote the complement of $A\subseteq E(G)$.

\subsubsection{Genus}

A ribbon graph is said to be {\em orientable} if it is orientable when viewed as a surface.  Similarly, the {\em genus}, $g(G)$, of a ribbon graph $G$ is its genus when viewed as a punctured surface. Note that the genus of a ribbon graph is the sum of the genera of its components. 

The genus of a surface is {\em not} additive under connected sums. (See Subsection~\ref{ss.bgpd}, just after Lemma~\ref{l1}, for a recap of the connected sum and some relevant facts on the topology of surfaces.) For example the connected sum of a torus and a real projective plane, which, are both surfaces of genus $1$, is homeomorphic to the connected sum of three real projective planes, a surface of genus $3$. To get around this technical difficulty, rather than writing our formulae in term of genus, we write it in terms of the {\em Euler genus}, $\gamma$, which is additive under the connected sum. If $G$ is a connected ribbon graph, then
\[\ga(G):=\left\{  \begin{array}{rl}  2g(G), & \text{if } $G$  \text{ is orientable};  \\  g(G), & \text{if } $G$  \text{ is non-orientable}.   \end{array}  \right. \]
If $G$ is not connected then, $\gamma(G)$ is defined as the sum of the value of $\gamma$ of each of its components.

We say that a ribbon graph $G$ is a {\em plane ribbon graph} if it is connected and $\gamma(G)=0$; and is a {\em \RP ribbon graph} if it is connected and $\gamma(G)=1$. (Note that here we insist that plane and \RP ribbon graphs are connected, which is not always the case in the literature.)

A cellularly embedded graph $G\subset \Sigma$ is a {\em plane graph} if $\Sigma$ is the  $2$-sphere, $S^2$; and is an {\em \RP graph} if $\Sigma$ is the real projective plane, \RPt. Plane ribbon graphs and plane cellularly embedded graphs correspond to one another, as do \RP ribbon  graphs and \RP cellularly embedded graphs.
This equivalence allows us to abuse notation and write `plane graph' for `plane ribbon graph', and `\RP graph' for `\RP ribbon graph'. This should cause no confusion.

We let $\chi(X)$ denote the Euler characteristic of a cellularly embedded graph,  ribbon graph or   surface $X$.

\subsubsection{Geometric duals}\label{ss.duals}

Let $G\subset \Sigma$ be a cellularly embedded graph. Recall that its {\em geometric dual} $G^*\subset \Sigma$ is the cellularly embedded graph obtained from $G$ by  placing one 
vertex in each of its faces, and embedding an edge of $G^*$ between two of these
vertices  whenever the faces of $G$ they lie in are adjacent. Edges of $G^*$ are embedded so that they cross the corresponding face boundary (or edge of $G$) transversally.   There is a natural bijection between the edges of $G$ and the edges of $G^*$. We  use this bijection to identify the edges of $G$ and the edges of $G^*$. Observe that $\gamma(G)=\gamma(G^*)$, and that duality acts disjointly on the components of a cellularly embedded graph.

Geometric duals have a particularly neat description in the language of ribbon graphs. Given a ribbon graph $G=(V(G), E(G))$, regard it as a punctured surface. Fill in the punctures using a set $V(G^*)$ of discs to obtain a closed surface. Delete the vertices in $V(G)$ from this surface. The resulting ribbon graph is the geometric dual $G^*=(V(G^*),E(G))$.   

Observe that if $G$ is an arrow-marked ribbon graph then, every marking arrow on a vertex of $G$ gives rise to a marking arrow on a vertex of $G^*$. We will use this observation later.

\subsection{Partial duality}\label{ss.pd}
In this subsection we describe partial duality and its basic properties. 

\begin{definition}[Chmutov \cite{Ch1}]\label{d.pd}
Let $G$ be a  ribbon graph and $A\subseteq E(G)$. Arbitrarily orient and label each of the edges of $G$ (the orientation need not extend to an orientation of the ribbon graph). The boundary components of the spanning ribbon subgraph $(V(G), A)$ of $G$ meet the edges of $G$ in disjoint arcs (where the spanning ribbon subgraph is naturally embedded in $G$). On each of these arcs, place an arrow which points in the direction of the orientation of the edge boundary  and is labelled by the edge it meets. 
The resulting marked boundary components of the spanning ribbon subgraph $(V(G), A)$ define an  arrow presentation. 
The  ribbon graph corresponding to this  arrow presentation is the {\em partial dual} $G^A$ of $G$.
\end{definition}

An example of a partial dual formed using Definition~\ref{d.pd} is shown in Figure~\ref{f.chpd}. In the figure, $G$ is an \RP ribbon graph, $A=\{3\}$, and $G^A$ is a non-orientable ribbon graph of genus $3$.

\begin{figure}
\centering
\subfigure[A ribbon graph $G$.]{
\includegraphics[scale=.4]{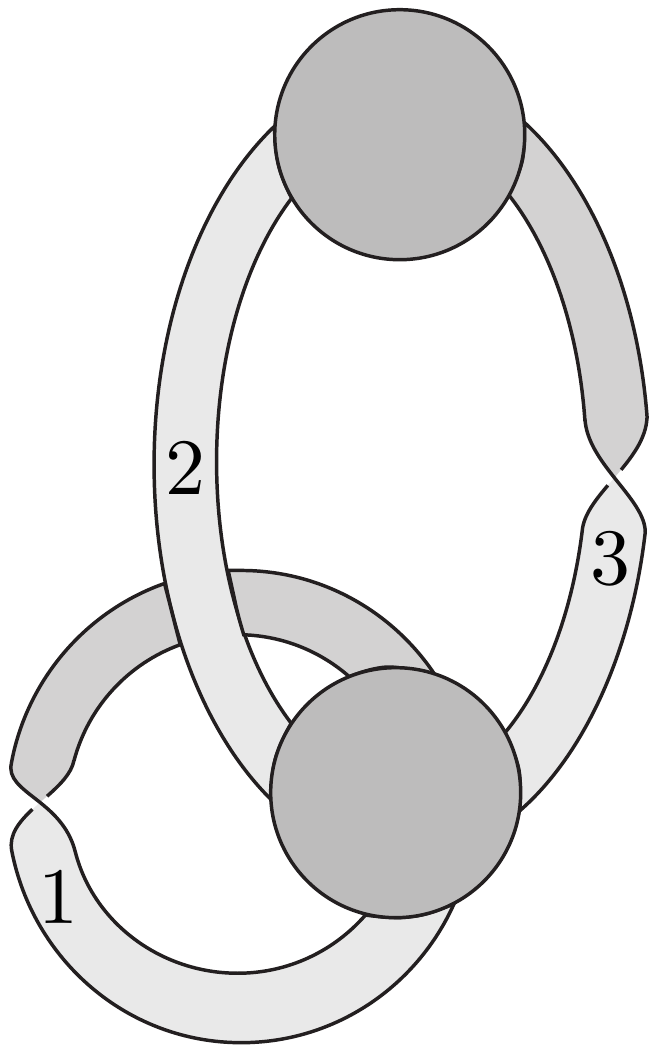}
\label{f.chpd1}
}
\hspace{10mm}
\subfigure[ The spanning ribbon subgraph $G-\{1,2\}$.]{
\includegraphics[scale=.45]{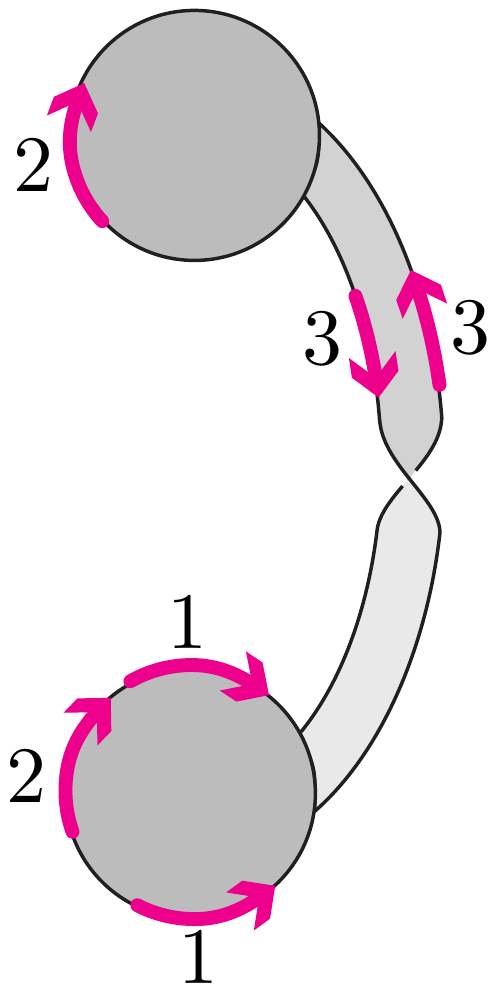}
\label{f.chpd2}
}
\hspace{10mm}
\subfigure[Its boundary component.]{
\raisebox{10mm}{\includegraphics[scale=.5]{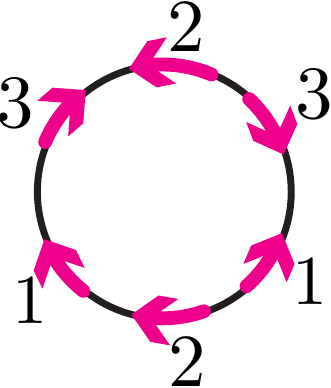}}
\label{f.chpd3}
}
\hspace{10mm}
\subfigure[The partial dual $ G^{\{3\}}$.]{
\raisebox{6mm}{\includegraphics[scale=.5]{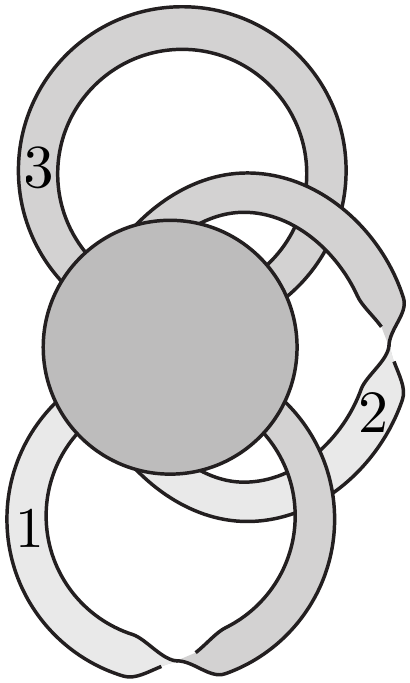}}
\label{f.chpd4}
}
\caption{Forming a partial dual using spanning ribbon subgraphs.}
\label{f.chpd}
\end{figure}

The idea behind a partial dual $G^A$ is to form the dual of $G$ with respect to only a subset $A$ of its edges. This can be achieved by deleting the edges in $A^c$ from $G$, recording their positions using marking arrows (giving $G\,\vec{-} \,A^c$); forming the geometric dual of this arrow-marked ribbon graph, retaining the marking arrows on the boundary (giving $\left(G\,\vec{-} \,A^c\right)^*$); and then obtaining $G^A$ by adding the edges in $A^c$ (giving $ \left(G\,\vec{-} \,A^c\right)^* \vec{+}\, A^c$). This gives:
\begin{proposition}[\cite{Mo4}]\label{p3}
Let $G$ be a ribbon graph and $A\subseteq E(G)$. Then 
\[ G^A:= \left(G\,\vec{-} \,A^c\right)^* \vec{+}\, A^c.\]
\end{proposition}
An example of the formation of a partial dual $G^A$ using Proposition~\ref{p3} is given in Figure~\ref{f.ampd}. In the figure, $G$ is an \RP ribbon graph, $A=\{3,4\}$, and $G^A$ is non-orientable and of genus $4$.

%\begin{center}
%\begin{tabular}{ccc}
%\includegraphics[height=40mm]{g6}   &\hspace{1cm} \raisebox{15mm}{=} \hspace{1cm}& 
%\includegraphics[height=40mm]{am1} \\
%$G$ & & $G\,\vec{-}A^c$ with $A^c=\{1,2,5\}$ \\
%&&\\
%\includegraphics[height=30mm]{pde1}   &\hspace{1cm} \raisebox{15mm}{=} \hspace{1cm} & 
%\includegraphics[height=40mm]{epd6} \\
%$(G\,\vec{-}A^c)^*$  & & $(G\,\vec{-}A^c)^*\,\vec{+}A^c = G^A$  
%\end{tabular}
%\end{center}

\begin{figure}
\centering
\subfigure[A ribbon graph $G$.]{
\includegraphics[scale=.5]{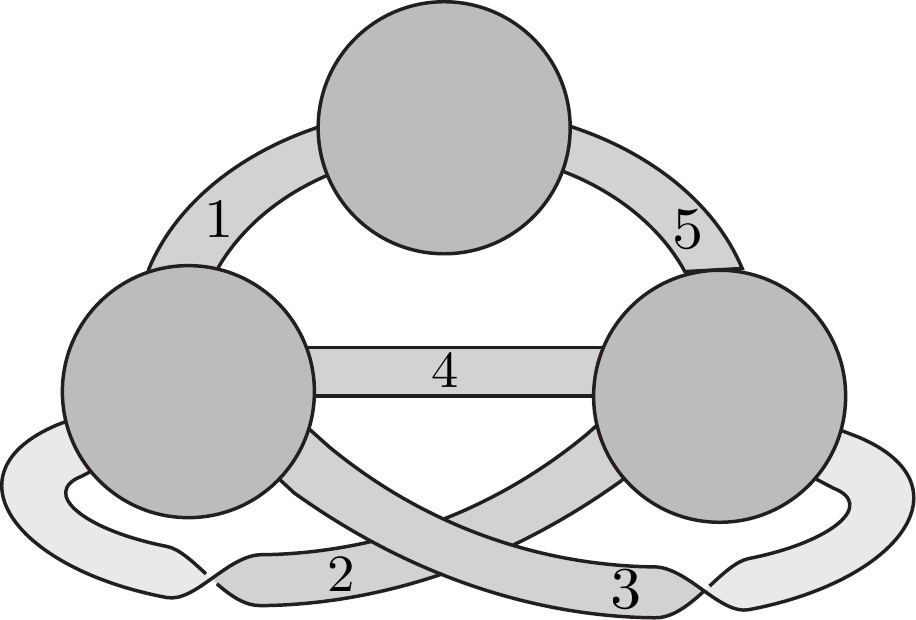}
\label{f.ampd1}
}
\hspace{10mm}
\subfigure[$G\,\vec{-}A^c$ with $A^c=\{1,2,5\}$. ]{
\includegraphics[scale=.5]{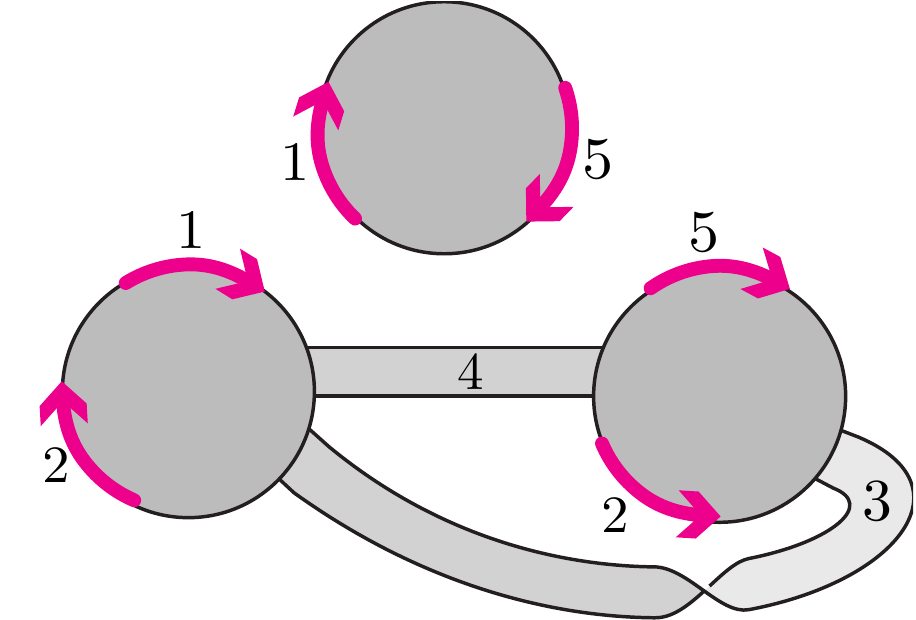}
\label{f.ampd2}
}\\
\subfigure[The geometric dual $(G\,\vec{-}A^c)^*$.]{
\includegraphics[scale=.5]{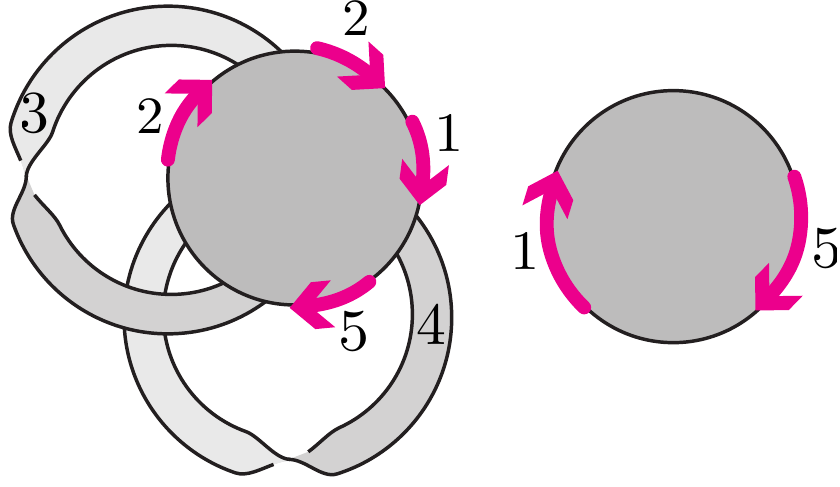}
\label{f.ampd3}
}
\hspace{10mm}
\subfigure[$(G\,\vec{-}A^c)^*\,\vec{+}A^c = G^{\{3,4\}}=G^A $.]{
\includegraphics[scale=.5]{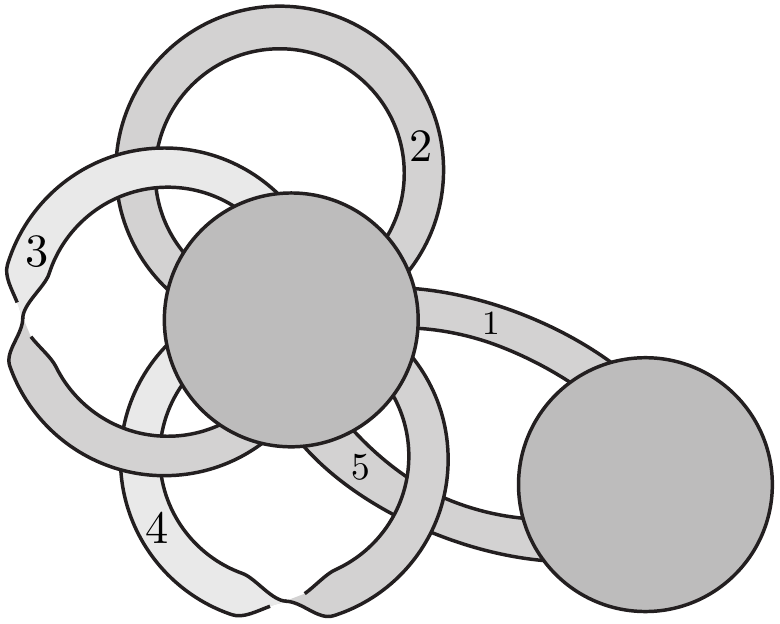}
\label{f.ampd4}
}
\caption{Forming a partial dual using arrow-marked ribbon graphs.}
\label{f.ampd}
\end{figure}

We will use the following basic properties of partial duals.
\begin{proposition}[Chmutov~\cite{Ch1}]\label{p.pd2}
Let $G$ be a ribbon graph and $A, B\subseteq E(G)$.  Then the following hold.
\begin{enumerate}
\item \label{p.pd2.1} $G^{\emptyset}=G$.
\item \label{p.pd2.2}  $G^{E(G)}=G^*$, where $G^*$ is the geometric dual of $G$.
\item \label{p.pd2.3} $(G^A)^B=G^{A\Delta B}$, where $A\Delta B := (A\cup B)\backslash  (A\cap B)$ is the symmetric difference of $A$ and $B$.
\item $G$ is orientable if and only if $G^A$ is orientable.
\item Partial duality acts disjointly on components, {\em i.e.} $ (P\sqcup Q)^A = (P^{A \cap E(P)} ) \sqcup (Q^{A \cap E(Q)})$.
\item Partial duals can be formed one edge at a time.
\item  There is a natural 1-1 correspondence between the edges of $G$ and the edges of $G^A$.
\end{enumerate}
\end{proposition}

\subsection{$n$-sums of ribbon graphs}
In this subsection we discuss $n$-sums of ribbon graphs, which   are   natural extensions of the corresponding operations for graphs.  In the next section we will use $1$-sums to introduce the concept of a $\bi$ of a ribbon graph and use it to determine the genus of a partial dual, providing a connection between the genus  and  separability.

\subsubsection{$n$-sums, $1$-sums, and joins}
We begin by describing $n$-sums and joins of ribbon graphs. These form the foundations of the decompositions of ribbon graphs considered here.

\begin{definition}\label{d.ns}
Let $G$ be a connected ribbon graph, $v_1, \ldots, v_n \in V(G)$, and let $P$ and $Q$ be non-trivial, connected ribbon subgraphs of $G$. Then $G$ is said to be an {\em $n$-sum} of $P$ and $Q$, written $G=P\oplus_n Q$, if $G=P\cup Q$ and $P\cap Q=\{v_1, \ldots, v_n\}$. (See Figure~\ref{f.sum}.) The $n$-sum is said to {\em occur} at the vertices $v_1, \ldots, v_n$. 
\end{definition}

\begin{figure}
\begin{tabular}{ccccc}
\includegraphics[height=2cm]{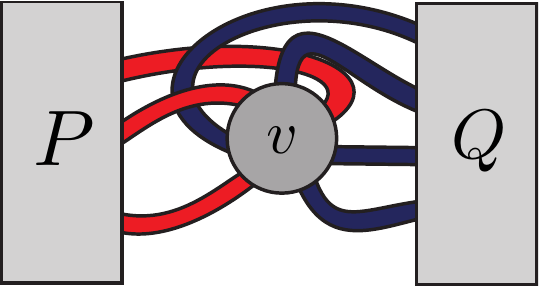} & \hspace{1cm} & \includegraphics[height=2cm]{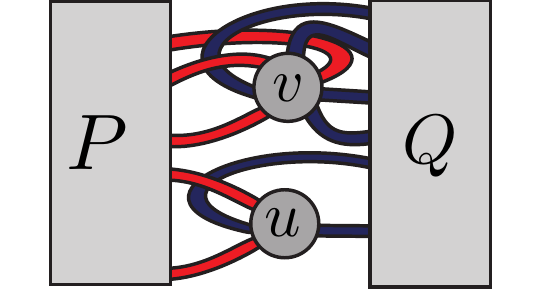} 
& \hspace{1cm} & \includegraphics[height=2cm]{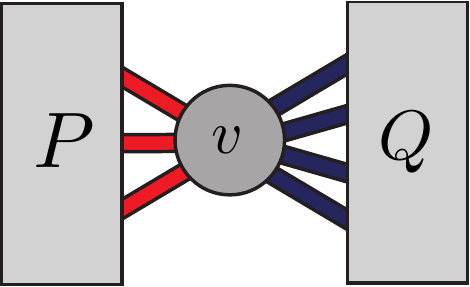}  \\
 A $1$-sum  $P\oplus Q$. && A $2$-sum  $P\oplus_2 Q$. &&A join $P\vee Q$.
\end{tabular}
\caption{A $1$-sum, a $2$-sum and a join of two ribbon graphs $P$ and $Q$.}
\label{f.sum}
\end{figure}

An $n$-sum, $G=P\oplus_n Q$, is defined as a decomposition of $G$ into ribbon subgraphs $P$ and $Q$. This means that we can, and will, identify the edges in $P$ and $Q$ with edges in $G$. Similarly, we can, and will, identify the vertices of $P$ and $Q$ with vertices of $G$.  The vertices $v_1, \ldots, v_n$ at which the $n$-sum occurs are the only vertices of $G$ that appear in both $P$ and $Q$. 

We can also view an $n$-sum as a way to construct a ribbon graph $G$ out of two connected ribbon graphs  $P$ and $Q$. Given $P$, $Q$,  vertices $v_1^P, \ldots, v_n^P\in V(P)$, and $v_1^Q, \ldots, v_n^Q\in V(Q)$. Then if, for each $i$, we identify $v_i^P$ and $v_i^Q$ in a way such that the edges incident to  $v_i^P$ and to $v_i^Q$ do not intersect, we obtain a ribbon graph $G$ that has the property that $G=P\oplus_n Q$, with the $n$-sum occurring at  $v_1, \ldots, v_n$. 
We  often find it convenient to regard an $n$-sum as such a commutative operation on ribbon graphs.

\medskip

Here we are especially interested in $1$-sums of ribbon graphs. We denote the $1$-sum operation, $\oplus_1$, simply by $\oplus$. We also note that the non-triviality requirement in Definition~\ref{d.ns} means that a $1$-summand never consists of an isolated vertex. This condition (and also the requirement that $P$ and $Q$ are connected) is for  convenience, and the results presented here can easily be adapted if it is dropped. Just as with abstract graphs, we say that a ribbon graph is {\em separable} if and only if it can be written as a $1$-sum of two ribbon graphs.

Another fundamental operation on ribbon graphs that is of interest here is the join. 
 The join, also known as  `one-point join',  `map amalgamation' and `connected sum' in the literature,  is a simple, special case of the $1$-sum.
 
 \begin{definition}
Suppose $G=P\oplus Q$ with the $1$-sum occurring at $v$. If  
there is an arc on the  boundary of $v$ with the property that all edges of $P$ incident to $v$ meet it on this arc, and no edges $Q$ do, then  $G$ is  the {\em join} of $P$ and $Q$, written $G=P\vee Q$. (See Figure~\ref{f.sum}.)
\end{definition}

%\begin{figure}
%\begin{tabular}{ccccc}
%\includegraphics[height=2cm]{j2} & \hspace{1cm} & \includegraphics[height=2cm]{j3} 
%& \hspace{1cm} & \includegraphics[height=2cm]{j1}  \\
%Ribbon graphs $P$ and $Q$. && A $1$-sum  $P\oplus Q$ && A join $P\vee Q$.
%\end{tabular}
%\caption{A $1$-sum and a join of two ribbon graphs.}
%\label{f.sum}
%\end{figure}

\subsubsection{Sequences of $1$-sums}\label{ss.seq}
An important observation is that, when regarded as an operation, the $n$-sum is not associative. Suppose that  $G=(P\oplus_{n_1} Q)\oplus_{n_2} R$, then, if $n_2\geq 2$, it is possible that the $\oplus_{n_2}$-sum involves vertices of both $P$ and $Q$, and so we can not write $G$ as $P\oplus_{n_1} (Q\oplus_{n_2} R)$. A second possibility is that the $n_2$-sum involves only vertices of $P$ (so $R\cap Q=\emptyset$), in which case the connectivity requirement in Definition~\ref{d.ns} means that we can not write $G$ as  $P\oplus_{n_1} (Q\oplus_{n_2} R)$. This second case applies  for $n_2\geq 1$, and since we are primarily interested in $1$-sums, is of more concern here.

   Although in general we can not write  $(P\oplus_{n_1} Q)\oplus_{n_2} R$ as $P\oplus_{n_1} (Q\oplus_{n_2} R)$, in certain cases we can. For example, if $P\cap R=\emptyset$ then  $(P\oplus_{n_1} Q)\oplus_{n_2} R$ can be written as $P\oplus_{n_1} (Q\oplus_{n_2} R)$. 
Note also that in such a case we can also write    $(P\oplus_{n_1} Q)\oplus_{n_2} R$ as $P\oplus_{n_1} (R\oplus_{n_2} Q)$, but not as $(P\oplus_{n_1} R)\oplus_{n_2} Q$.    
   
 With these observations on associativity in mind, we adopt the convention that    
\[   H_1 \oplus_{n_2} H_2\oplus_{n_3} H_3   \oplus_{n_4} \cdots\oplus_{n_l} H_l := (\cdots ((H_1 \oplus_{n_2} H_2)\oplus_{n_3} H_3)   \oplus_{n_4} \cdots\oplus_{n_l} H_l).\]

     We are particularly interested in expressions of $G$ as $1$-sums of ribbon graphs, and accordingly make the following definition.
  \begin{definition}  
We  say that $G$ can be {\em written as a sequence of $1$-sums} if $G$ contains subgraphs $H_1, \ldots ,H_l$ such that
\begin{equation}\label{e.seq}
G=   H_1 \oplus H_2\oplus H_3   \oplus \cdots\oplus H_l := (\cdots ((H_1 \oplus H_2)\oplus H_3)   \oplus \cdots\oplus H_l). 
\end{equation}
\end{definition}

\begin{example}\label{e.s}
Consider the ribbon graph graph $G$ shown in Figure~\ref{f.s1}. If we define the ribbon subgraphs 
$H_1=(\{a,b,c\},\{1,2,3,4\})$, 
$H_2=(\{a,d\},\{5\})$, 
$H_3=(\{b,e,f\},\{6,7,8\})$,
$H_4=(\{f\},\{9\})$,
$H_5=(\{c,g,h\},\{10,11,12,13\})$,   
$H_6=(\{h\},\{14\})$, and    
$H_7=(\{g\},\{15\})$, then we can write $G=H_1\oplus H_2\oplus\cdots \oplus H_7$, with the $1$-sums occurring at   $a$, $b$, $f$, $c$, $h$, and $g$, respectively.

The ribbon graph $G$ can also be written as, for example, $G=H_4 \oplus H_3 \oplus H_1\oplus H_5\oplus H_7\oplus H_2\oplus H_6$.  Also observe that  $G$ can not be expressed as a sequence of $1$-sums that starts with, for example, $H_1\oplus H_4$.
\end{example}

Observe that in \eqref{e.seq}, the $H_i$'s are  non-trivial ribbon subgraphs that cover $G$; that for each $i\neq j$, $H_i$ and $H_j$ have  at most one vertex in common; and that if a $1$-sum occurs at a vertex $v$ in the sequence, then $v$ is a separating vertex of the underlying abstract graph of $G$.

As discussed above, some reorderings of the terms in a sequence of $1$-sums is possible. 
We consider sequences of $1$-sums   to be {\em equivalent} if they differ only in the order of $1$-summation, and consider all sequences of $1$-sums up to this equivalence.

\begin{figure}
\centering
\subfigure[A ribbon graph $G$.]{
\includegraphics[scale=.5]{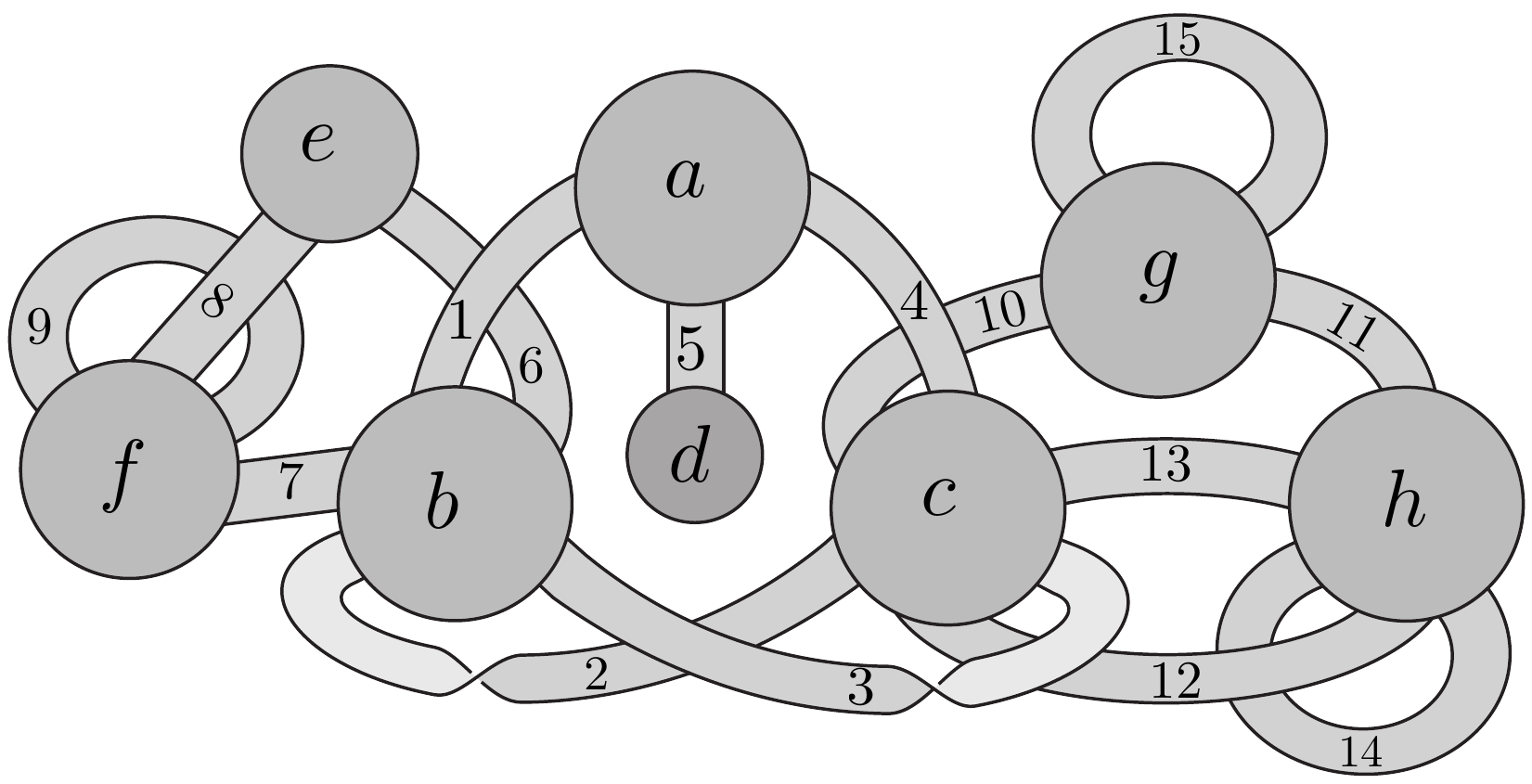}
\label{f.s1}
}
\hspace{11mm}
\subfigure[A graph associated with a sequence of $1$-sums.]{
\includegraphics[scale=0.7]{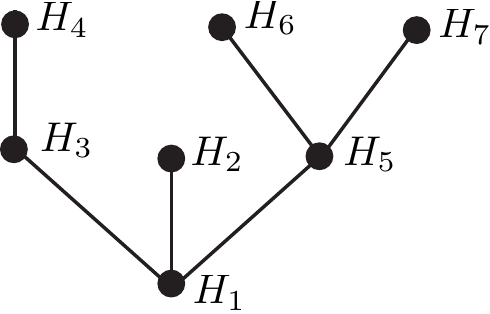}
\label{f.s2}
}
\label{f.s}
\caption{A ribbon graph that admits an \rpbit.}
\end{figure}

\bigskip

Sequences of $1$-sums have an associated graph  that can be used to reorder them. Suppose that 
$G=   H_1 \oplus H_2\oplus H_3   \oplus \cdots\oplus H_l$. Then we can associate a graph $\mathcal{T}$ with the sequence of $1$-sums by taking one vertex labelled $H_i$ for each ribbon subgraph $H_i$, and adding an edge between the vertices labelled $H_i$ and $H_j$ if and only if $H_i\cap H_j\neq \emptyset$. 

\begin{example} \label{e.tree}
The graph associated with the sequence of $1$-sums given in Example~\ref{e.s} is shown in Figure~\ref{f.s2}. 
\end{example}

We may use the graph $\mathcal{T}$ to reorder the sequence of $1$-sums as follows: choose a root of the graph and let $\mathcal{S}_1= H_i$, where $H_i$ is the label of the root. If $\mathcal{S}_j$ has been constructed, choose an $H_p$ that is not in  $\mathcal{S}_j$, but labels a vertex in $\mathcal{T}$ that is adjacent to one labelled by a summand in $\mathcal{S}_j$. Let $\mathcal{S}_{j+1}= \mathcal{S}_j \oplus H_p$. This results in a valid reordering $\mathcal{S}_l$ of the sequence of $1$-sums.
Since the choice of root is arbitrary, we have the following proposition:
\begin{proposition}\label{p.start} Let 
$G=   H_1 \oplus H_2\oplus H_3   \oplus \cdots\oplus H_l$. Then for each $i$, $G$ can be written as a sequence of $1$-sums in which $H_i$ is the first $1$-summand: $G=H_i\oplus H_{\iota_2} \oplus \cdots  \oplus H_{\iota_l}$.
\end{proposition}

\section{Separability and the genus of a partial dual}\label{s3}
In this section we prove the first of our main results which is a relation between the separability of a ribbon graph and the genus of a partial dual. We introduce the concept of a \bi of a ribbon graph which, loosely speaking, says that the ribbon graph can be constructed by $1$-summing the elements from two sets of ribbon graphs together  in such a way that no $1$-sum involves two components from the same set. We will see that the genus of a partial dual is determined by the genera of the summands in a \bit. We will use this result later  to completely characterize the partial duals of low genus ribbon graphs.

\subsection{Biseparations}\label{ss.bi}
Let $G$ be a ribbon graph and $A$ be a non-empty, proper subset of $E(G)$. The set $A$ and its complementary subset $A^c$ partition  $E(G)$, and  induce (not necessarily connected) ribbon subgraphs $G|_A$ and $G|_{A^c}$  of $G$. Every component of $G|_A$ and of $G|_{A^c}$ 
can be regarded as a subgraph embedded in $G$, and we can therefore write 
\begin{equation}\label{e.ssbi1}
G= H_1 \oplus_{n_2} H_2 \oplus_{n_3} \cdots \oplus_{n_l} H_l, 
\end{equation}
where each $H_i$ is a unique component of  $G|_A$ or of $G|_{A^c}$, and every component of $G|_A$ and $G|_{A^c}$ appears as an $H_i$. (To obtain \eqref{e.ssbi1}, choose a component of $G|_A$ or of $G|_{A^c}$ and keep summing components of $G|_A$ and $G|_{A^c}$, so that the resulting ribbon graph is connected, until each component is used.) Furthermore, observe that by the  construction of the $H_i$, every $n_i$-sum in Equation~\eqref{e.ssbi1} involves one component of $G|_A$ and one of $G|_{A^c}$.  If each $n_i$-sum in \eqref{e.ssbi1} is a $1$-sum then we say that $A$ defines a \bit. Formally: 
\begin{definition}\label{d1}
Let $G=(V,E)$ be a connected ribbon graph and $A\subseteq E$.    We say that $A$ defines a {\em \bi} if either
\begin{enumerate}
\item  $A=E$ or $A=\emptyset$ (in which case the \bi is {\em trivial}); or,
\item $G$ can be written as a sequence of $1$-sums in which  each $1$-sum  involves a component of $G|_A$ and a  component of $G|_{A^c}$. 
\end{enumerate} 
\end{definition}
The {\em length} of a non-trivial \bi is the length of its sequence of $1$-sums, and the length of a trivial \bi is  $1$.

\begin{example}\label{e.bis} Some examples of \bis are given below.
\begin{enumerate}
\item\label{e.bis1} The sets $A=\{1\}$ and  $A=\{2,3\}$ both define non-trivial \bis of the ribbon graph  shown in Figure~\ref{f.chpd1}.
\item \label{e.bis2} Every subset of $E(G)=\{1,2,3\}$ defines a \bi of the ribbon graph  shown in Figure~\ref{f.chpd4}.
\item\label{e.bis3} Only $\emptyset$ and $E(G)$ define \bis of the ribbon graph in Figure~\ref{f.ampd1}. 
\item\label{e.bis5} For the ribbon graph in Figure~\ref{f.ampd4}, $A$ defines a \bi if and only if it contains either both $1$ and $5$, or neither $1$ nor $5$. 
\item\label{e.bis4} For the ribbon graph $G$ in Figure~\ref{f.s1}, let 
$A_1=\{1,2,3,4\}$, 
$A_2=\{5\}$, 
$A_3=\{6,7,8\}$,
$A_4=\{9\}$,
$A_5=\{10,11,12,13\}$,   
$A_6=\{14\}$, and    
$A_7=\{15\}$.
Then $A$ defines a \bi of $G$ if and only if $A=\bigcup_{i\in I} A_i$, for some $I\subseteq \{1,\ldots, 7\}$.
\end{enumerate}
\end{example}

Observe that in Definition~\ref{d1} (and in the preceding discussion) there is no distinction between  $A$ and $A^c$. This means that $A$ defines a \bi if and only if $A^c$ does. For reference later we record this observation as a proposition:
\begin{proposition}\label{p2}
Let $G$ be a connected ribbon graph. Then $A$ defines a \bi of $G$ if and only if $A^c$ does.
Moreover, if the \bi is non-trivial, then $A$ and $A^c$ define \bi with the same set of sequences of $1$-sums.
\end{proposition}

It is also worthwhile observing that if $A$ defines a non-trivial \bit, then every $1$-sum in the sequence of $1$-sums it defines occurs at a different vertex of $G$. We also note that, as before, we identify the edges and vertices of the subgraphs $G|_A$ and $G|_{A^c}$ with those of $G$ in the natural way.

Note that the graph associated with a biseparation (as described at the end of Section~\ref{ss.seq}) is a tree.

\subsection{Biseparations and the genus of a partial dual}\label{ss.bgpd}
We  come to the first of our main results. This result provides a connection between the genus of a partial dual and separability.

\begin{theorem}\label{t1}
Let $G$ be a connected ribbon graph and $A\subseteq E(G)$. Then $A$ defines a biseparation of $G$ if and only if 
\[ \ga(G^A) = \ga(G|_A)+ \ga(G-A). \]
Furthermore, if $A$ defines a \bit, then $G^A$ is orientable if and only if both $G|_A$ and  $G-A$ are.
\end{theorem}

Note that $\ga(G|_{A^c})=\ga(G-A)$, and  that $G|_{A^c}$ is orientable if and only if   $G-A$ is. Thus   Theorem~\ref{t1} can be expressed in terms of $G|_A$ and $G|_{A^c}$. Here, however, we prefer  to work in terms of $A$ rather than $A^c$.

The remainder of this section is taken up with the proof of  Theorem~\ref{t1}.  We begin with a   proposition and lemma that concern the ways in which partial duality interacts with $n$-sums.

\begin{proposition}\label{p1}
Let $G$ be a ribbon graph.
\begin{enumerate}
\item If $v\in V(G)$ is an isolated vertex in $G-A^c$, then $v$ is also a vertex of $G^A$.

\item Suppose $G=P\oplus_n Q$ with the $n$-sum occurring at $v_1, \ldots, v_n$. Then every vertex of $V(P)\bs \{v_1, \ldots, v_n\}$ is also a vertex of  $G^{E(Q)}= (P\oplus_n Q)^{E(Q)}$.

\item If $G=P\oplus_{n_1} Q\oplus_{n_2}R$, with $P\cap R=\emptyset$ and $A\subseteq E(P)$, then $G^A =   (P\oplus_{n_1} Q)^A  \oplus_{n_2}R $.

\end{enumerate}

\end{proposition}
\begin{proof}
For the first item, if $v$ is an isolated vertex of $G-A^c$, then it is also one of $G\vec{-}A^c$.
As geometric duality acts disjointly on components, and the geometric dual of an isolated  vertex is an isolated vertex, $v$ is also an isolated vertex in $(G\vec{-}A^c)^*$. It follows that $v$ is a vertex of $(G\vec{-}A^c)^*\vec{+} A^c$ which by Proposition~\ref{p3} is $G^A$.

The second item follows from the first as the elements of  $V(P)\bs \{v_1, \ldots, v_n\}$ are all isolated vertices of $G-E(P)$.

For the third item, begin by observing that every vertex in the ribbon subgraph $R$ of $G$ is an isolated vertex  in $G\vec{-}A^c$, and that no marking arrows in $G\vec{-}A^c$ labelled by edges in $E(R)$ lie on the same vertex as marking arrows labelled by edges in $E(P)$ (as $P\cap R=\emptyset$). Thus  $R$ and $(P\oplus_{n_1} Q)^A$ are both ribbon subgraphs of $G^A=  (P\oplus_{n_1} Q\oplus_{n_2}R)^A$, and these ribbon subgraphs intersect in exactly $n_2$ vertices. It follows that  $G^A =   (P\oplus_{n_1} Q)^A  \oplus_{n_2}R $.
\end{proof}

\begin{lemma}\label{l1}
Let $P$ and $Q$ be ribbon graphs. Then
\begin{enumerate}
\item \label{l1.1}  $\chi((P\oplus_n Q)^{E(Q)})=\chi((P\oplus_n Q)^{E(P)})=  \chi(P)+\chi(Q)-2n$;
\item \label{l1.2}  $\ga((P\oplus Q)^{E(Q)})=\ga((P\oplus Q)^{E(P)})= \ga(P)+\ga(Q)$;
\item \label{l1.3}   $\ga((P\oplus_n Q)^{E(Q)})=\ga((P\oplus_n Q)^{E(P)})> \ga(P)+\ga(Q)$, when $n\geq 2$.
\end{enumerate}
\end{lemma}
For the proof of the Lemma we recall a few basic facts about the classification of surfaces. We let  $T^2$ denote the torus, and \RP the real projective plane. The connected sum,  $\Sigma  \# \Sigma' $,  of two surfaces $\Sigma$ and  $\Sigma' $ is obtained by deleting the interior of a disc in each surface and identifying the two boundaries.
We have $\mathbb{R}\mathrm{P}^2 \#T^2=\mathbb{R}\mathrm{P}^2\#\mathbb{R}\mathrm{P}^2\#\mathbb{R}\mathrm{P}^2$, and  $\mathbb{R}\mathrm{P}^2\#\mathbb{R}\mathrm{P}^2$ is the Klein bottle.
A handle is an annulus $S^1\times I$, where $S^1$ is a circle and $I$ is the unit interval. By adding a handle to $\Sigma$, we mean that we remove the interiors of two discs from $\Sigma$, and identify each boundary component of the punctured surfaces with a distinct boundary component of $S^1\times I$.  Adding a handle to $\Sigma$ either connect sums a torus or a Klein bottle to $\Sigma$, depending upon how it is attached.

\begin{figure}
\centering
\subfigure[A ribbon graph \mbox{$G=P\oplus_3 Q$.}]{
\includegraphics[scale=.7]{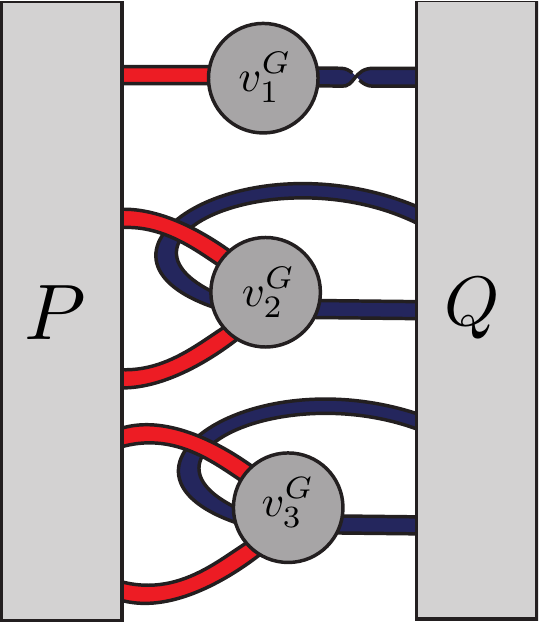}
\label{f.pl1}
}
\hspace{10mm}
\subfigure[The 3-summands $P$ and $Q$. ]{
\includegraphics[scale=.7]{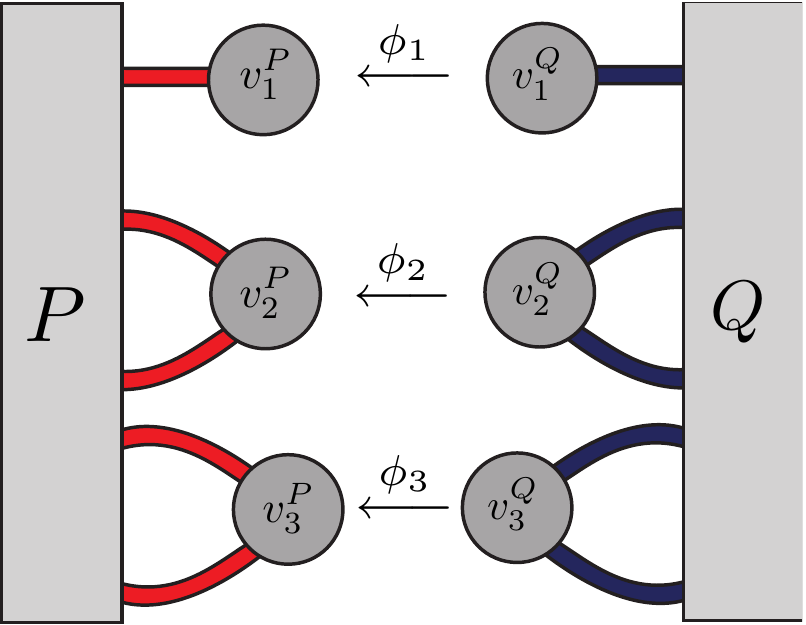}
\label{f.pl2}
}\\
\subfigure[Cellular embeddings  $P\subset \Sigma_P$ and $Q\subset \Sigma_Q$.]{
\includegraphics[scale=.6]{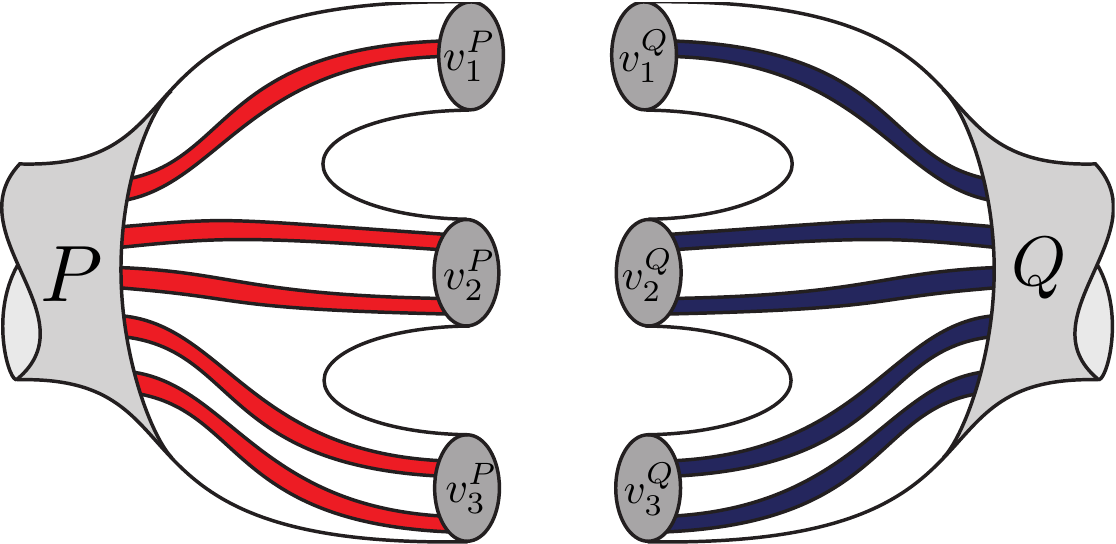}
\label{f.pl3}
}
\hspace{10mm}
\subfigure[Cellular embeddings $P^*\subset \Sigma_P$ and $Q\subset \Sigma_Q$.]{
\includegraphics[scale=.6]{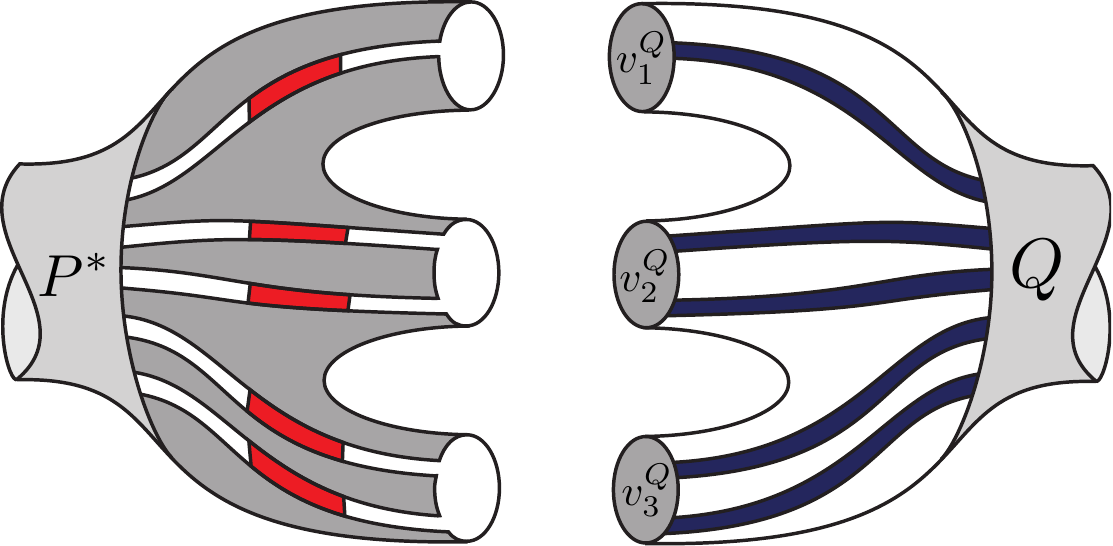}
\label{f.pl4}
}
\hspace{10mm}
\subfigure[A cellular embedding of $(P\oplus_n Q)^{E(P)}$.]{
\includegraphics[scale=.7]{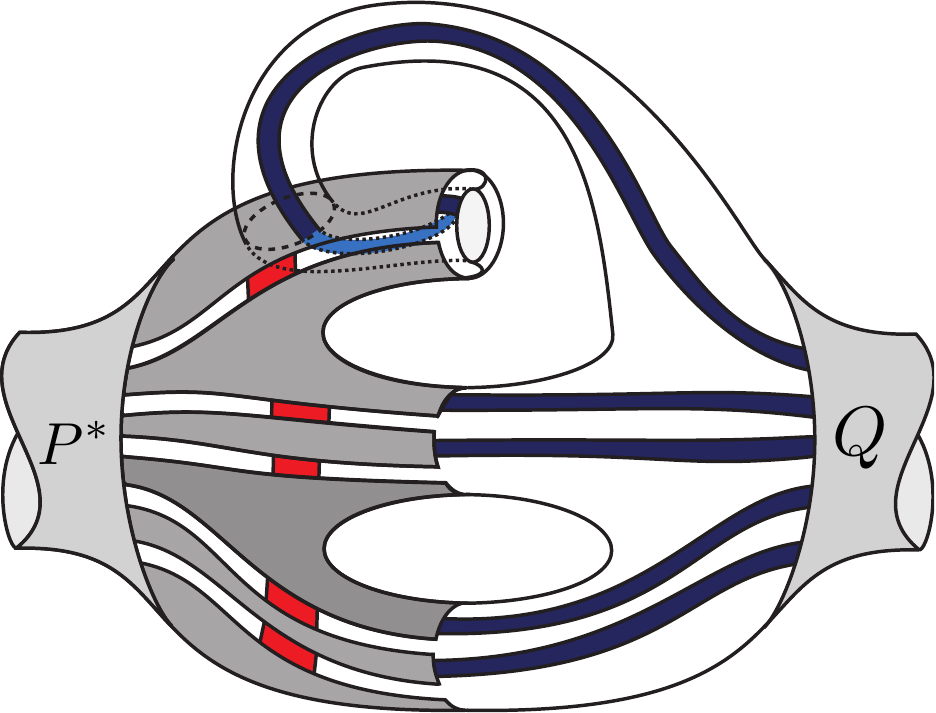}
\label{f.pl5}
}
\caption{Obtaining a cellular embedding of $(P\oplus_n Q)^{E(P)}$ as in the proof of Lemma~\ref{l1}.}
\label{f.pl}
\end{figure}

\begin{proof}[Proof of Lemma~\ref{l1}]
The first equality in each of the three items follows  since  $g((P\oplus_n Q)^{E(Q)})=g((P\oplus_n Q)^{E(P)})^*)= g((P\oplus_n Q)^{E(Q)})$, or by using the fact that $P\oplus_n Q=Q\oplus_n P$.

For the remaining identities, our strategy is to use cellular embeddings of $P$ and  $Q$ to construct a cellular embedding of the partial dual $(P\oplus_n Q)^{E(P)}$, counting the number of handles we need to add to obtain this cellular embedding. The construction of the cellular embedding described below is illustrated in Figure~\ref{f.pl}.

Since $\ga$, $\chi$ and partial duality act disjointly on connected components, we can, and will, assume without loss of generality that $G:=P\oplus_n Q$ is connected. It follows that $P$ and $Q$ are also connected.
Suppose that the $n$-sum occurs at the vertices $v_1^G,\ldots , v_n^G$ of $G$.   
Then there are vertices $v_1^P,\ldots , v_n^P$ of $P$, and $v_1^Q,\ldots , v_n^Q$ of $Q$, such that $G=P\oplus_n Q$ is obtained by identifying $v_i^P$ and $v_i^Q$, for each $i$. (See Figures~\ref{f.pl1} and~\ref{f.pl2}.)
Suppose also that for each $i$, $\phi_i: v_i^Q\rightarrow v_i^P$ is the mapping that identifies   $v_i^P$ and $v_i^Q$, and that this identification results in the vertex $v_i^G$ of $G$.  Let $\phi_i|_{\partial}$, denote the restriction of $\phi_i$ to the boundary of the vertices. 

Cellularly embed $P$ into a surface $\Sigma_P$, and $Q$ into a surface $\Sigma_Q$. (See Figure~\ref{f.pl3}.)
Below, we will need to keep track of the location of the vertices $v_1^P,\ldots , v_n^P$ on the surface $\Sigma_P$. To do this, for $i=1,\ldots ,n$, let $\fD_i^P$ denote the disc on $\Sigma_P$  on which   $v_i^P$ lies, so $\fD_i^P:=\Sigma_P \cap v_i^P$. Note that $\phi_i|_{\partial}$ maps the boundary of $v_i^Q$ to the boundary of $\fD_i^P$.

Construct an embedded ribbon graph $H$ as follows: 
\begin{enumerate}
\item Form the geometric dual $P^*$ of $P$, embedding $P^*$ in the natural way in $\Sigma_P$ so that the vertices of $P\subset \Sigma_P$ form the faces of $P^*\subset \Sigma_P$ and vice versa. (See Figure~\ref{f.pl4}.)
\item Noting that the $\fD_i^P$ are faces of $P^*\subset \Sigma_P$, delete the interior of each disc $\fD_i^P$ and the interior of each vertex $v_i^Q \subset \Sigma_Q$. Identify the boundaries of the resulting punctured surfaces according to the mappings  $\phi_i|_{\partial}$. (See Figure~\ref{f.pl5}.)
\end{enumerate} 
This process results in an embedding of a ribbon graph $H$ in a surface $\Sigma_H$ (since, under the identification,  each  arc on an edge  of $Q$ that meets a $v_i^Q$ is attached to an arc on the boundary of a vertex of $P^*$, and this arc does not meet an edge of $P^*$).
Each component of  $ \Sigma_H \bs H $ corresponds to  either a component of  $\Sigma_P\bs P^*$, of  $\Sigma_Q\bs Q$, or is obtained by merging such components. $P^*$ and $Q$ are both cellularly embedded. As $P^*$ is cellularly embedded, no face of $P^*$ touches a disc $\fD_i^P$ twice (otherwise it is not a disc). As the faces of $P^*$ correspond to  the vertices of $P$, no face of $P^*$ can touch more than one of the discs $\fD_i^P$. It then follows that each of the components of  $ \Sigma_H \bs H $ that arise by merging faces is a disc, and therefore each face of $ \Sigma_H \bs H $ is a disc. Thus $H$ is cellularly embedded in $\Sigma_H$.

%
%Next we show that $H=(P\oplus_n Q)^{E(P)}$.  Arbitrarily label and orient each edge of $G=P\oplus_n Q$. Wherever an edge meets a vertex $v^G_i$, place  an arrow that points in the direction determined by the edge orientation and labelled with the label of the edge. From this decorated ribbon graph, construct a decorated ribbon graph $\vec{P}$  by deleting all of the edges in $E(Q)$, and then deleting any isolated vertices. Similarly,  construct an arrow marked ribbon graphs $\vec{Q}$  by deleting all of the edges in $E(P)$ and then deleting any isolated vertices. Note that $P\oplus_n Q$ is  obtained from $\vec{P}$ and $\vec{Q}$ by  identifying marking arrows of the same label, and that each $\phi_i$ identifies the marking arrows of the same label. It is readily seen that the partial dual $ (P\oplus_n Q)^{E(P)} =  [(P\oplus_n Q) \vec{-} E(Q)]^* \vec{+} E(Q) $, is obtained from  $\vec{P}$ and $\vec{Q}$ by forming $\vec{P}^*$ and then identifying marking arrows of the same label. But this is exactly the ribbon graph $H$ since the $\phi_i|_{\partial}$ ensure that the arrows are identified correctly.
%
 
 Next we show that $H=(P\oplus_n Q)^{E(P)}$. To do this we decorate the ribbon graph $G$ with labelled arrows then follow the construction of $H$ focussing on what happens to the ribbon graphs.  Arbitrarily label and orient each edge of $G=P\oplus_n Q$. Wherever an edge meets a vertex $v^G_i$, place  an arrow that points in the direction determined by the edge orientation and labelled with the label of that edge. From this decorated ribbon graph, construct a decorated ribbon graph $\vec{P}$  by 
deleting all of the edges in $E(Q)$, and then deleting any isolated vertices. Similarly,  construct a decorated ribbon graph $\vec{Q}$  by deleting all of the edges in $E(P)$ and then deleting any isolated vertices. (Note that the maps $\phi_i$ that recover $G$ from $P$ and $Q$ by identifying vertices also apply to $\vec{P}$ and $\vec{Q}$. Thus $G$ is recovered from $\vec{P}$ and $\vec{Q}$ by identifying vertices using $\phi_i$. Observe that when $\phi_i$ identifies $v_i^P$ and $v_i^Q$ the arrows on the vertices are identified. Moreover, the map $\phi_i$ can be completely determined by matching up the arrows on $v_i^P$ and $v_i^Q$ so that the labels and orientations match, and then extending the identification to rest of the vertex in the obvious way.) 
 The construction of the (non-embedded) ribbon graph $H$ can can be described in the following way:   delete the vertices $v_i^Q$ of $\vec{Q}$ (but not their incident edges) and, for each $i$, identify each arrow that was on $v_i^Q$ with the arrow on  $\vec{Q}$ of the same label (this describes the identifications under the $\phi_i|_{\partial}$). But, as the dual of an isolated vertex is an isolated vertex, this is just a description  of   $[(P\oplus_n Q) \vec{-} E(Q)]^* \vec{+} E(Q) $ in which the edges are attached in a particular order. It follows by Proposition~\ref{p3} that  $H=(P\oplus_n Q)^{E(P)}$, as required.

So far we have shown that $H=(P\oplus_n Q)^{E(P)}$ is cellularly embedded in $\Sigma_H$. It remains to determine the surface $\Sigma_H$.
To do this observe that  $\Sigma_H$ can be obtained by: (1) starting with $\Sigma_P$ and $\Sigma_Q$, deleting the interiors of $v_1^Q$ and $\fD_1^P$, and identifying their boundaries to form $\Sigma_P\# \Sigma_Q$; (2) deleting the interiors of $v_2^Q$ and $\fD_2^P$, and identifying their boundaries, adding a handle to  $\Sigma_P\# \Sigma_Q$; (3) repeating this step for $i=3,\ldots, n$, adding a further $n-2$ handles.
Thus $\Sigma_H$ is obtained by adding $n-1$ handles to  $\Sigma_P\# \Sigma_Q$. We then have $\chi(\Sigma_H) =  \chi(\Sigma_P)+\chi(\Sigma_Q)-2n$, giving the first item in the lemma.
Also, as adding a handle to a surface corresponds to connect summing it with either a torus or a Klein bottle, we have 
\[ \Sigma_H =  \Sigma_P\# \Sigma_Q  \#  \underbrace{T^2  \# \cdots  \#  T^2}_a \#   \underbrace{\mathbb{R}\mathrm{P}^2  \#  \cdots \#  \mathbb{R}\mathrm{P}^2}_{2b}, \]
where $a+2b=n-1$. 
(For example, in  Figure~\ref{f.pl}, $\Sigma_H= \Sigma_P\# \Sigma_Q  \#  T^2 \#  \mathbb{R}\mathrm{P}^2\#  \mathbb{R}\mathrm{P}^2$.)  
%(Note that the above expression can be written wholly in terms of either tori or real projective planes depending upon orientability.)
Thus  $\gamma (\Sigma_H) \geq  \ga(\Sigma_P)+\ga(\Sigma_Q)$, with equality if and only if $n=1$. The second and third items of the lemma follow. 
\end{proof}

\begin{proof}[Proof of Theorem~\ref{t1}]
First  suppose that $A$ determines a \bi of $G$. 
 We will prove that  $\ga(G^A) = \ga(G|_A)+ \ga(G-A)$ by induction on the length of a \bit.

 If $A$ determines a \bi of a ribbon graph of length $1$ the result is trivial. If $A$ determines a \bi of a ribbon graph of length $2$, the result follows immediately from Lemma~\ref{l1}.

For the inductive step, assume that the assertion holds for all ribbon graphs and edge sets that define a \bi of length less than $l$.
Suppose that $G$ is a ribbon graph and that $A\subseteq E(G)$ defines a \bi of $G$ of length $l\geq 3$. As $A$ defines a non-trivial \bit, we have that  
\begin{equation}\label{et13}G= H_1 \oplus H_2 \oplus \cdots \oplus H_l , \end{equation}
where the $H_i$ are in 1-1 correspondence with the components of $G|_A$ and $G|_{A^c}$, and every $1$-sum occurs at a different vertex and involves a component of $G|_A$ and $G|_{A^c}$. 
Note that, as $\gamma$ is additive over  components, and as isolated vertices are of genus zero, 
\begin{equation}\label{et14}
\ga(H_1)+ \cdots +\ga(H_{l})=  \ga(G|_A)+ \ga(G|_{A^c})
=  \ga(G|_A)+ \ga(G-A).
\end{equation}
As $H_l$ is the last subgraph in the sequence of $1$-sums in Equation~\eqref{et13},  exactly one $1$-sum involves a vertex of $H_l$. Thus we may write
\begin{equation}\label{et11}
 G= H_1 \oplus H_2 \oplus \cdots  \oplus (H_i \oplus H_l)  \oplus \cdots       \oplus H_{l-1}.
  \end{equation}
There are now two cases to consider: when $E(H_l)\subseteq A$, and when $E(H_l)\not\subseteq A$. First suppose that $E(H_l)\subseteq A$, and so $E(H_i)\not\subseteq A$. Then  
\begin{equation}\label{et120}
G^{E(H_l)} %= (H_1 \oplus H_2 \oplus \cdots \oplus H_{l-1} \oplus H_l )^{E(H_l)} \\
  = (H_1 \oplus H_2 \oplus \cdots  \oplus (H_i \oplus H_l)  \oplus \cdots       \oplus H_{l-1} )^{E(H_l)}    = (H_1 \oplus H_2 \oplus \cdots  \oplus (H_i \oplus H_l)^{E(H_l)}  \oplus \cdots       \oplus H_{l-1} ),
\end{equation}
where we have used the facts that, by Proposition~\ref{p1}, if the $1$-sum $H_i \oplus H_l$ occurs at $v$ then all other vertices of $H_i$ are also vertices of $ (H_i \oplus H_l)^{E(H_l)}$, and that every $1$-sum in \eqref{et11} occurs at a different vertex.
By noting that $(H_i \oplus H_l)^{E(H_l)}$ can be regarded as a single $1$-summand in this sequence, we see that  $G^{E(H_l)}$ can be written as a sequence of $1$-sums of length $l-1$; with each $1$-summand involving a component of $G^{E(H_l)}|_{A\bs E(H_l)}$ and of $G^{E(H_l)}|_{(A\bs E(H_l))^c }$.  (Note that $(H_i \oplus H_l)^{E(H_l)}$ is a component of $G^{E(H_l)}|_{(A\bs E(H_l))^c }$ since $E(H_l)\subseteq A$ and  $E(H_i)\not\subseteq A$.)  Thus $A\bs E(H_l)$ defines a \bi of $G^{E(H_l)}$ of length $l-1$. The inductive hypothesis then gives
\begin{equation}\label{et12}
\ga((G^{E(H_l)})^{A\bs E(H_l)}) = 
\ga( (G^{E(H_l)})|_{A\bs E(H_l)} )+  \ga( G-( A\bs E(H_l) )  )
\end{equation}
 But as $\gamma$ is additive over  components, and isolated vertices are of genus zero, we can use \eqref{et120} to rewrite \eqref{et12} as
 \begin{multline*}
\ga((G^{E(H_l)})^{A\bs E(H_l)}) = 
\ga(H_1)+ \cdots \ga((H_i \oplus H_l)^{E(H_l)}) +\cdots +\ga(H_{l-1}) \\
= \ga(H_1)+ \cdots \ga(H_i) +\cdots +\ga(H_{l-1})+\ga(H_{l}) 
=  \ga(G|_A)+ \ga(G-A),
\end{multline*}
where the second equality follows by Lemma~\ref{l1}, and the third by Equation~\eqref{et14}.
Finally, by Proposition~\ref{p.pd2},  $ \ga(G^A)= \ga((G^{E(H_l)})^{A\bs E(H_l)})$, and the result follows, completing the case where $E(H_l)\subseteq E(G)$.

Now suppose that  $E(H_l)\not\subseteq A$. By Proposition~\ref{p2}, as $A$ defines a \bi of $G$ with its sequence of $1$-sums \eqref{et13},  the complementary subset $A^c$ also defines a \bi of $G$ with sequence of $1$-sums given by \eqref{et13}. Moreover, $E(H_l)\subseteq A^c$ and so the previous case gives that 
\begin{equation}\label{et15} \ga(G^{A^c})=  \ga(G|_{A^c})+ \ga(G-{A^c}). \end{equation}
Using Proposition~\ref{p.pd2}, we have $\ga(G^{A^c})= \ga((G^{A^c})^*) =  \ga(G^{A})$. Also, $ \ga(G|_{A^c})=   \ga(G-{A})$, and $\ga(G-{A^c})=\ga(G|_A)$. Substituting these into \eqref{et15} gives $\ga(G^{A})=  \ga(G|_{A})+ \ga(G-{A})$, completing the proof of the `if' case of Theorem~\ref{t1}.

\medskip

For the converse, let $G$ be a ribbon graph and $A\subseteq E(G)$ be such that $\ga(G^A) = \ga(G|_A)+ \ga(G-A)$. 
If $A=E(G)$ or $A=\emptyset$, the result is trivial, so assume that this is not the case.
Then, since $G|_A$ and $G|_{A^c}$  partition the edge set of $G$, we can write
\begin{equation}\label{et16}
G= H_1 \oplus_{n_2} H_2 \oplus_{n_3} \cdots \oplus_{n_l} H_l, 
\end{equation}
where $H_1, \ldots, H_l$ are the components of $G|_A$ and $G|_{A^c}$, and where each $n_i$-sum involves one component of $G|_A$ and one of $G|_{A^c}$ (see the discussion at the beginning of Subsection~\ref{ss.bi}).
 To prove the theorem, we need to show that in \eqref{et16} each $n_i=1$.

Either  $E(H_l)\subseteq A$ or $E(H_l)\not\subseteq A$.
First suppose that $E(H_l)\subseteq A$. As $H_l$ is the last ribbon subgraph in the sequence of $n_j$-sums in \eqref{et16},  exactly one $n_j$-sum involves vertices of $H_l$. Suppose that $H_l$ is $n_l$-summed to $H_i$. 
 Then $E(H_i)\not\subseteq A$. Thus, by Propositions~\ref{p.pd2} and ~\ref{p1}, we can write
 \begin{equation}\label{et18} 
G^A = (G^{A\bs E(H_l)})^{E(H_l)} 
 =  ( (H_1 \oplus_{n_2} H_2 \oplus_{n_3} \cdots  \oplus_{n_{l-1}} H_{l-1}  )^{A\bs E(H_l)}  \oplus_{n_l} H_l  )^{E(H_l)}.
 \end{equation}
Lemma~\ref{l1} then gives that if $E(H_l)\subseteq A$ then
 \begin{equation}\label{et19} 
\ga(G^A)\geq 
 \ga (( H_1 \oplus_{n_2} H_2 \oplus_{n_3} \cdots  \oplus_{n_{l-1}} H_{l-1}  )^{A\bs E(H_l)} ) +\ga( H_l  ),
 \end{equation}
with  equality if and only if $n_l=1$.

On the other hand, if $E(H_l)\not\subseteq A$, then $E(H_l)\subseteq A^c$, by arguing as before we can write
\[ 
G^{A^c} = (G^{{A^c}\bs E(H_l)})^{E(H_l)} 
 =  (( H_1 \oplus_{n_2} H_2 \oplus_{n_3} \cdots  \oplus_{n_{l-1}} H_{l-1}  )^{A^c\bs E(H_l)}  \oplus_{n_l} H_l  )^{E(H_l)}.
\]
Then Lemma~\ref{l1} gives that 
\begin{equation}\label{et110}   \ga(G^{A^c})\geq   \ga (( H_1 \oplus_{n_2} H_2 \oplus_{n_3} \cdots  \oplus_{n_{l-1}} H_{l-1}  )^{A^c\bs E(H_l)})  +\ga( H_l  )   \end{equation}
with  equality if and only if $n_l=1$.
However, using Proposition~\ref{p.pd2}, $\ga(G^{A^c}) = \ga((G^{A^c})^*)=\ga(G^A)$.
Also 
\[\ga (( H_1 \oplus_{n_2}\cdots  \oplus_{n_{l-1}} H_{l-1}  )^{A^c\bs E(H_l)})
= \ga ((( H_1 \oplus_{n_2}  \cdots  \oplus_{n_{l-1}} H_{l-1}  )^{A^c\bs E(H_l)})^*)
= \ga ((H_1 \oplus_{n_2}  \cdots  \oplus_{n_{l-1}} H_{l-1}  )^{A}),
\]
where the last equality uses the facts that $H_l$ is not a summand and $E(H_l)\not\subseteq A$.
Equation~\eqref{et110} then gives that if $E(H_l)\not\subseteq A$, then
\begin{equation}\label{et111}   \ga(G^{A})\geq   \ga (( H_1 \oplus_{n_2} H_2 \oplus_{n_3} \cdots  \oplus_{n_{l-1}} H_{l-1}  )^{A})  +\ga( H_l  )   \end{equation}
with  equality if and only if $n_l=1$. (Note that in \eqref{et111} the exponent $A$, can be written as $A\bs E(H_l)$.)

Finally, repeated applications of Equations~\eqref{et19} and~\eqref{et111} then give  
\begin{multline*}
 \ga(G^{A})\geq  \ga (( H_1 \oplus_{n_2} H_2 \oplus_{n_3} \cdots  \oplus_{n_{l-1}} H_{l-1}  )^{A\bs E(H_l)})  +\ga( H_l  )   \\
 \geq  \ga (( H_1 \oplus_{n_2} H_2 \oplus_{n_3} \cdots  \oplus_{n_{l-2}} H_{l-2}  )^{A\bs E(H_l) \bs E(H_{l-1} )})  +\ga( H_{l-1}  )   +\ga( H_l  )  \\
 \geq \cdots\geq   \ga(H_1)+\ga(H_2)+\cdots  +\ga( H_l  ),
 \end{multline*}
with equality if and only if $n_1=n_2=\cdots =n_l=1$, as required.
%
%We will do this by induction on $l$. 
%
%First suppose that $l=2$. Then, by Item~\ref{l1.3} of Lemma~\ref{l1}, $n_2=1$ as, otherwise, 
 %$\gamma(G^A)> \ga(G|_{A})+ \ga(G-{A})$.
%
%For the inductive step, assume that the assertion holds for all ribbon graphs and all its edge sets that define a \bi of length less than $l$. There are two cases to consider: when   $E(H_l)\subseteq A$, and when $E(H_l)\not\subseteq A$.
%
%
% First suppose that $E(H_l)\subseteq A$.  As $H_l$ is the last subgraph in the sequence of $n_j$-sums in \eqref{et16},  exactly one $n_j$-sum involves a vertex of $H_l$. Suppose that $H_l$ is $n_l$-summed to $H_i$. 
% Then $E(H_i)\not\subseteq A$. Thus, by Propositions~\ref{} and ~\ref{p1}, we can write
% \begin{equation}\label{et18} 
%G^A = (G^{A\bs E(H_l)})^{E(H_l)} 
% =  (( H_1 \oplus_{n_2} H_2 \oplus_{n_3} \cdots  \oplus_{n_{l-1}} H_{l-1}  )^{A\bs E(H_l)}  \oplus_{n_l} H_l  )^{E(H_l)}.
 %\end{equation}
%Lemma~\ref{l1} then gives that if 
%\begin{equation}\label{et19} 
%\ga(G^A)\geq 
% \ga (( H_1 \oplus_{n_2} H_2 \oplus_{n_3} \cdots  \oplus_{n_{l-1}} H_{l-1}  )^{A\bs E(H_l)}  +\ga( H_l  ),
% \end{equation}
%with  equality if and only if $n_l=1$.
%
% By the inductive hypothesis,  
\end{proof}

\section{Characterizing the partial duals of low genus ribbon graphs}\label{s4}
In this section we apply Theorem~\ref{t1} to obtain a characterization of partial duals of plane graphs, and of \RP graphs, in terms of the existence of a \bit. To do this we introduce the concepts of  \pbis and  \rpbist, which are \bis with a restriction on the topology of the ribbon  subgraphs  $G|_A$ and $G|_{A^c}$. We show, in Theorem~\ref{t2}, that \pbis and  \rpbis characterize the partial duals of plane graphs and \RP graphs, respectively. We then go on, in Theorem~\ref{t.tog}, to  relate all of the  \pbis and  \rpbis  that a ribbon graph can admit.

\subsection{A characterization of plane and \RP partial duals}
We begin  with the observation that if $A$ defines a \bi of $G$ in which  one component  of $G|_A$ or of $G|_{A^c}$ is of genus $g$, and all of the  others are plane, then, by  Theorem~\ref{t1},  the partial dual $G^A$ is also of genus $g$. 
Motivated by this, we make the following definitions.
\begin{definition}\label{d2}
Let $G$ be a connected ribbon graph and $A\subseteq E(G)$. Then we say that
\begin{enumerate}
\item $A$ defines a {\em \pbi} if $A$ defines a \bi in which every component of $G|_A$ and of $G-A$ is plane;
\item $A$ defines a {\em \rpbi} if $A$ defines a \bi in which exactly one component of   $G|_A$ or of $G-A$ is \RP and all of the other components are plane.
\end{enumerate}
\end{definition}

\begin{example}\label{e.rpbis}
 Some examples of \rpbis and \pbis are given below. In these examples we focus on \rpbist, referring the reader to \cite{Mo5} for additional examples of \pbist. The examples given below should be compared with the examples of \bis given in Example~\ref{e.bis}.

\begin{enumerate}
\item\label{e.rpbis1}   Only $\emptyset$ and $E(G)$ define \rpbis of the ribbon graph  shown in Figure~\ref{f.chpd1}  ({\em cf.} Item~\eqref{e.bis1} of Example~\ref{e.bis}).

\item \label{e.rpbis2} Only $\{1,2\}$ and $\{3\}$ define  \rpbis of the ribbon graph  shown in Figure~\ref{f.chpd4} ({\em cf.} Item~\eqref{e.bis2} of Example~\ref{e.bis}).

\item\label{e.rpbis3} Only $A$ and $E(G)$ define \rpbis of the ribbon graph in Figure~\ref{f.ampd1} ({\em cf.} Item~\eqref{e.bis3} of Example~\ref{e.bis}).

\item\label{e.rpbis6}  For the ribbon graph in Figure~\ref{f.ampd4}, only $\{3,4\}$ and $\{1,2,5\}$ define \rpbis ({\em cf.} Item~\eqref{e.bis5} of Example~\ref{e.bis}).

\item\label{e.rpbis4} For the ribbon graph $G$ in Figure~\ref{f.s1}, 
let $B_1=\{ 1,2,3,4,9,14 \}$ and
$B_2=\{ 6,7,8,  10,11,12,13\}$.
Then  only the following sets define  \bist: $B_1$, $B_1\cup \{5\}$, $B_1\cup \{15\}$, $B_1\cup \{5,15\}$, and the complementary sets $B_2$, $B_2\cup \{5\}$, $B_2\cup \{15\}$, $B_2\cup \{5,15\}$ ({\em cf.} Item~\eqref{e.bis4} of Example~\ref{e.bis}).

\item\label{e.rpbis5}  For an example of a \pbit,  let $G$ be the ribbon graph  in Figure~\ref{f.s1}, and  let $G'$ be the orientable ribbon graph obtained from $G$ by removing the half-twists from the edges $2$ and $3$. Then the sets from Item~\ref{e.rpbis4} are exactly those that define \pbis of $G'$.
\end{enumerate}
\end{example}

We now come to our second main result which is a characterization of partial duals of plane and \RP graphs in terms of \pbis and \rpbist, respectively. The plane case in Theorem~\ref{t2} was first appeared in \cite{Mo5}, however the proof given here is new. 
\begin{theorem}\label{t2}
Let $G$ be a connected ribbon graph and $A\subseteq E(G)$. Then
\begin{enumerate}
\item \label{t2.1} $G^A$ is a plane ribbon graph if and only if $A$ defines a \pbi of $G$;
\item \label{t2.2}  $G^A$ is an \RP ribbon graph if and only if $A$ defines an \rpbi of $G$.
\end{enumerate}
\end{theorem}

\begin{proof}
~

\noindent\underline{\eqref{t2.1} $\impliedby$:}
If $A$ defines a \pbi of $G$ then $\gamma(G_A)+\gamma(G|_A)=0$, and so, by Theorem~\ref{t1}, $\gamma(G^A) =0$, and so $G^A$ is a plane graph.

\medskip

\noindent\underline{\eqref{t2.2} $\impliedby$:} Similarly, if $A$ defines an \rpbi of $G$ then $\gamma(G_A)+\gamma(G|_A)=1$, and so, by Theorem~\ref{t1}, $\gamma(G^A) =1$, and so $G^A$ is an \RP graph.

\medskip

\noindent\underline{\eqref{t2.1} $\implies$:} 
  Suppose that $G^A$ is a plane ribbon graph. We need to show that $A$ defines a \pbit. \ 
 If $A=\emptyset$ or $A=E(G)$ the result is trivial, so assume this is not the case.
 
We will show that all of the components of $G|_A$ and $G-A$ are plane, so $\ga(G|_A)+\ga(G-A)=0$. Also, since $G^A$ is plane, $\ga(G^A)=0$. It then follows from Theorem~\ref{t1} that $A$ defines a \bi of $G$, and since  all of the components of $G|_A$ and $G-A$ are plane, this is a \pbit.
 
 We first show that $\ga(G|_A)=0$. Observe that, as $G^A$ is plane, all of the components of $G^A-A^c$ are also plane. Using Proposition~\ref{p2}, we have
\begin{equation}\label{et2.1}
 G^A-A^c =  [\left(G\,\vec{-} \,A^c\right)^* \vec{+}\, A^c]-A^c=\left(G\,- \,A^c\right)^* .
 \end{equation}
 Since geometric duality acts disjointly on connected components and preserves the genus of a ribbon graph, it follows that  every component of $G-A^c$, and therefore of $G|_A$, is plane.

We now show that $G-A$ is plane. Our argument is illustrated in Figure~\ref{f.a}.
 Cellularly embed $G^A$ in $S^2$. We will consider $G^A-A$ and $G^A-A^c$ as ribbon subgraphs embedded in $G^A\subset S^2$. 
 Then $S^2\bs\Int (G^A-A^c)$ (where $\Int$ denotes the interior) is a collection of punctured and non-punctured discs, {\em i.e.} it is a collection of  punctured spheres. (See Figure~\ref{f.a2}.) Observe that all of the edges of $G^A$ that belong to $A^c$ are embedded in these punctured spheres, and each embedded edge belonging to  $A^c$ meets the boundary  of exactly one of the punctured spheres in exactly two arcs. 
 For each punctured sphere $D$, form a plane ribbon graph $H_D$ by filling in the punctures of $D$ with discs that form the vertex set $V_D$ of $H_D$. Let $E_D$ be the edges of $A^c$ that lie in $D$, and let $H_D=(V_D,E_D)$. As this construction gives an embedding of each $H_D$ in $S^2$, it follows that each $H_D$ is plane. Let $H$ be the union of the $H_D$. (See  Figure~\ref{f.a4}.) Note that $\ga(H)=0$.

\begin{figure}
\centering
\subfigure[A plane ribbon graph $G^A$.]{
\includegraphics[height=4.5cm]{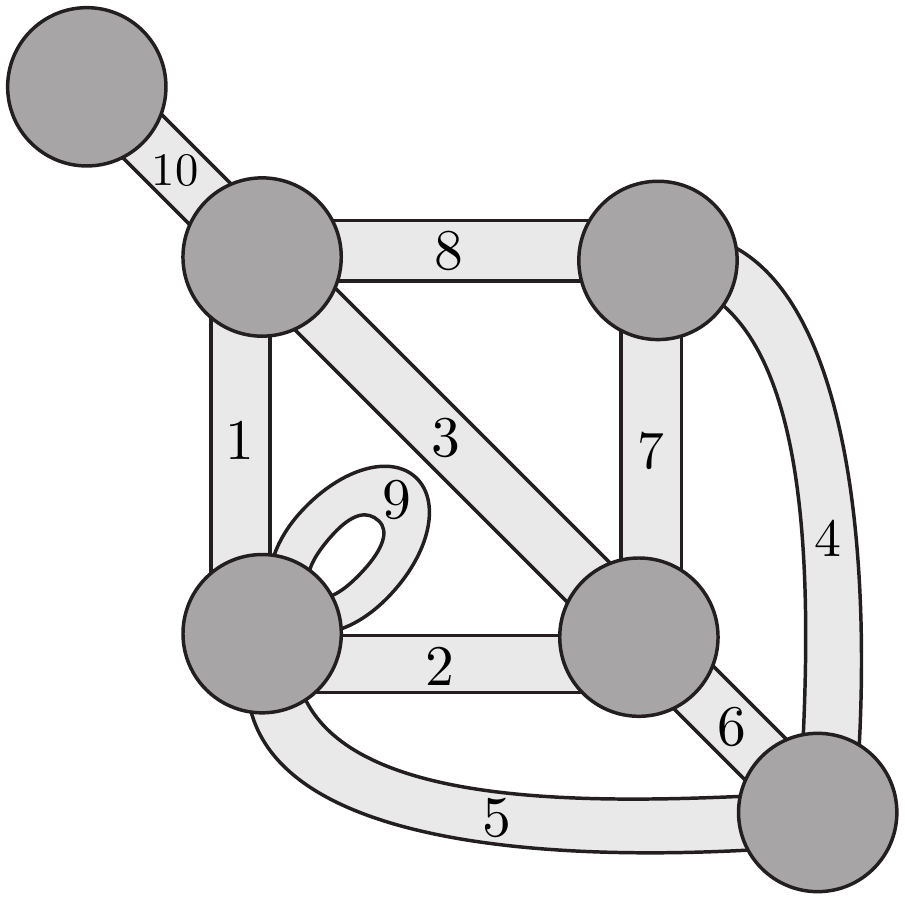}
\label{f.a1}
}
\hspace{11mm}
\subfigure[ $S^2\bs\Int (G^A-A^c)$, with $A=\{1,2,3,4\}$, giving  edges embedded in  a once punctured sphere and a thrice punctured sphere.]{
\includegraphics[height=5.5cm]{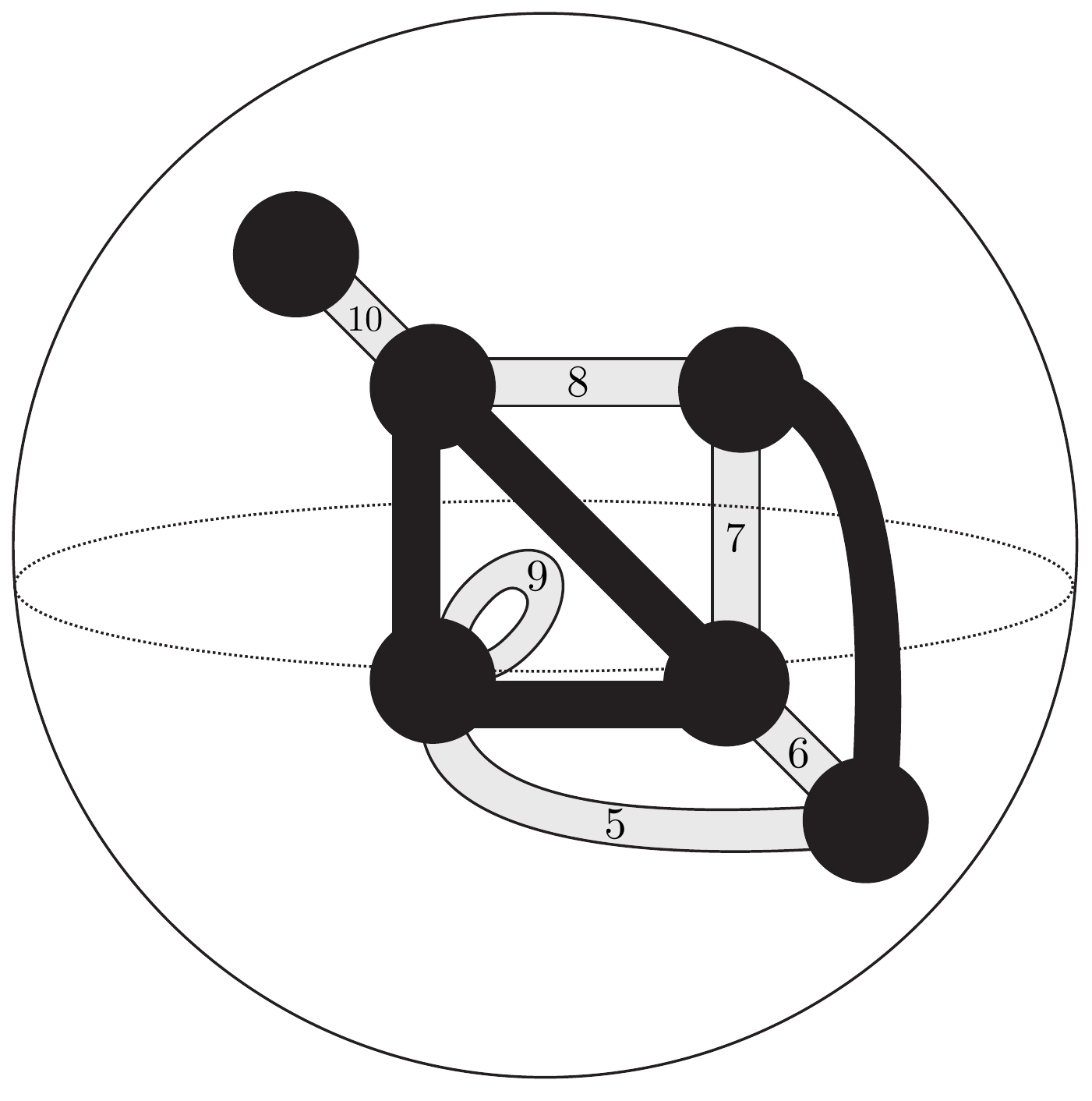}
\label{f.a2}
}
\hspace{11mm}
\subfigure[The resulting ribbon graph $H$.]{
\includegraphics[height=4.5cm]{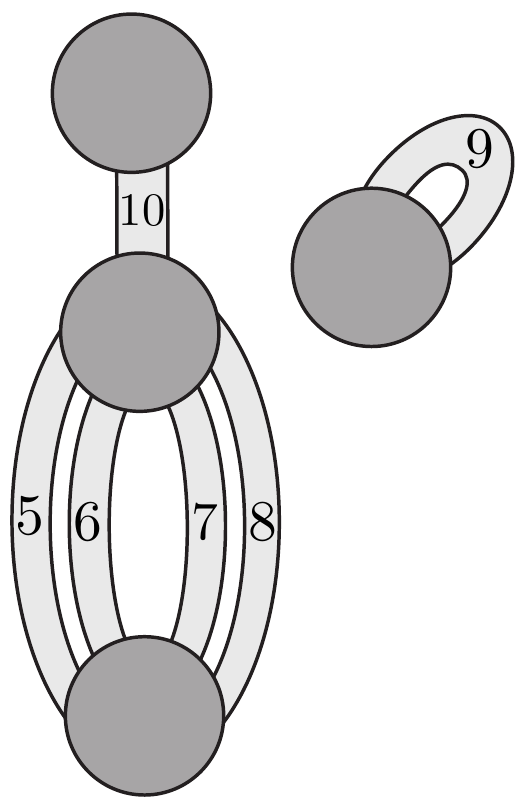}
\label{f.a4}
}
\caption{The construction of $H=(G^A)^A-A$ from the proof of Item~\ref{t2.1} of Theorem~\ref{t2}.}
\label{f.a}
\end{figure}

The ribbon graph $H$ is   $(G^A)^A-A$.  To see why this is, consider the construction of $(G^A)^A-A$ using Definition~\ref{d.pd}. 
We begin by arbitrarily orienting and labelling each of the edges of $G$. We add labelled arrows to the boundary components of $G^A-A^c$ using the labelling and orientation of the edges to obtain an arrow presentation for $(G^A)^A$. To obtain  $(G^A)^A-A$ from this, delete all of the arrows labelled by elements of $A$. This process can be simplified by adding only the arrows labelled by elements of $A^c$ to the boundary components of  $G^A-A^c$. The resulting arrow presentation clearly describes $H$. Thus $H=(G^A)^A-A$, but as $(G^A)^A=G$, we have that $H=G-A$ and $\ga (G-A)=\ga(H)=0$.

  Thus   $\gamma(G|_A)+\gamma(G-A)=0=\gamma(G^A)$, and so, using Theorem~\ref{t1}, $A$ defines a \bi of $G$ in which every component of $G|_A$ and of $G-A$ is plane. It follows that $A$ defines a \pbit, as required.

\medskip

\noindent\underline{\eqref{t2.2} $\implies$:} Our approach to the \RP case is similar, but more involved, to that of the plane case above.

Suppose that $G^A$ is an \RP graph.  We need to show that $A$ defines an \rpbi of $G$. If $A=\emptyset$ or $A=E(G)$ the result is trivial, so assume this is not the case.

We show that $G|_A$ and $G-A$ have exactly one \RP component between them, and all of the other components are plane. From this it follows that  $\gamma(G_A)+\gamma(G|_A)=1=\gamma(G^A)$, and so $A$ defines a \bi of $G$ which, by the genus and orientability of the connected components of $G|_A$ and $G-A$, must be an \rpbit.

There are two cases to consider: when $G^A-A^c$ has an \RP component, and when it does not. 

\noindent{\underline{Case 1:}} Our argument is straightforward adaption of the proof of the plane case of the theorem given above. Cellularly embed $G^A$ in \RPt.
Suppose that $G^A-A^c$ has an \RP component. Since all non-contractible cycles in an \RP graph intersect,  $G^A-A^c$ has exactly one \RP component. Denote this component by $K$.
Since all of the other components of $G^A-A^c$ must lie in the faces of $K\subset G^A\subset \rp$, they must all be plane. By Equation~\eqref{et2.1}, 
 $(G-A^c)^*$  must then  have exactly one \RP component and  all of the others must be plane. 
Since geometric duality acts disjointly on connected components, preserving genus and orientability, it follows that  $G-A^c$, and so $G_A$, has exactly one \RP component and  all of the others are plane. Thus $\ga(G_A)=1$.

It remains to show that all of the components of $G-A$ are plane. To do this, cellularly embed $G^A$ in \RP and regard $G^A-A$ and $G^A-A^c$ as being embedded in $G^A\subset \rp$.
Let $K$ be the \RP component of $G^A-A^c$ (which exists by hypothesis). 
 As all  non-contractible cycles in an \RP graph intersect, $\rp\bs \Int(K)$ is a collection of discs. Since $K$ is a component of $G^A-A^c\subset \rp$, it follows  that $\rp \bs \Int (G^A-A^c)$ is a collection of punctured spheres. Observe that all of the edges of $G^A$ that belong to $A^c$ are embedded in these punctured spheres, and each embedded edge belonging to  $A^c$ meets the boundary  of exactly one of the punctured spheres in exactly two arcs. 
 For each punctured sphere $D$, form a plane ribbon graph $H_D$ by filling in the punctures of $D$ with discs that form the vertex set $V_D$ of $H_D$. Let $E_D$ be the edges of $A^c$ that lie in $D$, and let $H_D=(V_D,E_D)$. As this construction gives an embedding of each $H_D$ in $S^2$, it follows that each $H_D$ is plane. Let $H$ be the union of the $H_D$.

The ribbon graph $H$ is   $(G^A)^A-A$: to obtain an arrow presentation of $(G^A)^A-A$, as in Definition~\ref{d.pd}, arbitrarily orient and label  each edge of $G^A$, add labelled arrows to  the boundary components of $G^A-A^c$ as described in Definition~\ref{d.pd}, but only for the edges in $A^c$ (as we only want an arrow presentation for $(G^A)^A-A$). The resulting arrow presentation clearly describes $H$ and so $H= (G^A)^A-A=G-A$. Thus $\ga (G-A)=\ga(H)=0$.

 As   $\gamma(G|_A)+\gamma(G|_A)=1=\gamma(G^A)$,  Theorem~\ref{t1} gives that $A$ defines a \bi of $G$, and this \bi is an \rpbit.

\noindent{\underline{Case 2:}} Suppose that $G^A-A^c$ does not have  an \RP component. (Note that $G^A-A$ may or may not have an \RP component.)
It follows that every component of $G^A-A^c$ is plane. By Equation~\eqref{et2.1},  each component of  $(G^A-A^c)^*$, and therefore of $G^A-A^c$, and so $G|_A$ is plane. 

It remains to show that $G-A$ has exactly one \RP component and all of the others are plane. Cellularly embed $G^A$ in \RP and  regard the components of $G^A-A^c$ and  $G^A-A$ as embedded ribbon subgraphs of $G^A \subset \rp$. To obtain the components of $(G^A)^A-A$ from $G^A$ (via Definition~\ref{d.pd})  arbitrarily orient and label  each edge of $G^A$, add labelled arrows to  the boundary components of $G^A-A^c$, as described in Definition~\ref{d.pd}, but only for the edges in $A^c$ (as we only want an arrow presentation for $(G^A)^A-A$). The arrow marked boundary components give an arrow presentation for $(G^A)^A-A$. To obtain the ribbon graph $(G^A)^A-A$, fill in the boundary cycles to form vertices of the ribbon graphs, and add the edges in the way prescribed by the labelled arrows, as in Figure~\ref{arrows} and Subsection~\ref{ss.ap}.

The following provides an alternative description of this construction of $(G^A)^A-A$. This construction is illustrated in Figure~\ref{f.b}. Start with the boundary cycles of $G^A-A^c\subset G^A\subset \rp$. Denote the set of boundary cycles by $\mathcal{C}$. Take the union of $\mathcal{C}$ with all of the embedded edges in $A\subset G^A \subset \rp$. This defines a set of cycles in \RP with embedded ribbon graph edges between them (see Figure~\ref{f.b2}). Denote this set by $\mathcal{G}$. Then $(G^A)^A -A$ is obtained from $\mathcal{G}$ by forming a ribbon graph by placing each cycle in $\mathcal{G}$ on the boundary of a disc, which becomes the vertex of a ribbon graph (see Figure~\ref{f.b3}).  

In $\mathcal{G}$, there is exactly one element $K$ that contains a non-contractible (topological) cycle (since all  non-contractible (topological) cycles in  \RP intersect). $K\subset \rp$ is cellularly embedded and so   gives rise to an \RP component of the ribbon graph in $(G^A)^A -A$.  
All of the components in $\mathcal{G}\bs K$ lie in $\rp \bs K$, which is a set of discs (as $K$ contains a non-contractible cycle). This means that every element  of $\mathcal{G}\bs K$ can be cellularly embedded in a disc, and so each one  gives rise to a plane ribbon graph. Thus $(G^A)^A -A$, and so $G-A$, contains exactly one \RP component and all of the others are plane.  
It then follows that     $\gamma(G|_A)+\gamma(G|_A)=1=\gamma(G^A)$.  Theorem~\ref{t1} then gives that $A$ defines a \bi of $G$, and this \bi must be an \rpbit, completing the proof of the theorem.
%
%These correspond to the  vertices of $(G^A)^A - A$. Next add  edges between the boundary cycles by following the edges in  of $A^c$ in $G^A\subset \rp$. These edges correspond to the edges of $(G^A)^A - A$.  Denote this system of embedded cycles and edges by $\mathcal{G}$. Note that every connected component in $G$ is embedded, and that the components of $(G^A)^A - A$ are obtained from the connected components in $\mathcal{G}$ by filling in the cycles to form vertices. 
%
%In $\mathcal{G}$, there is exactly one connected component that contains a non-contractible (topological) cycle (since all  non-contractible (topological) cycles in  \RP intersect). By removing all other connected components from $\mathcal{G}$ an filling in the cycles to form discs, we obtain an embedding of an \RP  component of $(G^A)^A - A$.  If $A\in \mathcal{G}$ is any other connected component. Then it is embedded  within a contractible disc in \RP and corresponds to a plane component of $(G^A)^A - A$. Thus $(G^A)^A - A=G-A$ has exactly one \RP component and all of the others are plane, as required, completing the proof that if $G^A$  id \RP, then $A$ defines an \rpbit.
\end{proof}

\begin{figure}
\centering
\subfigure[An \RP ribbon graph $G^A$.]{
\includegraphics[height=3cm]{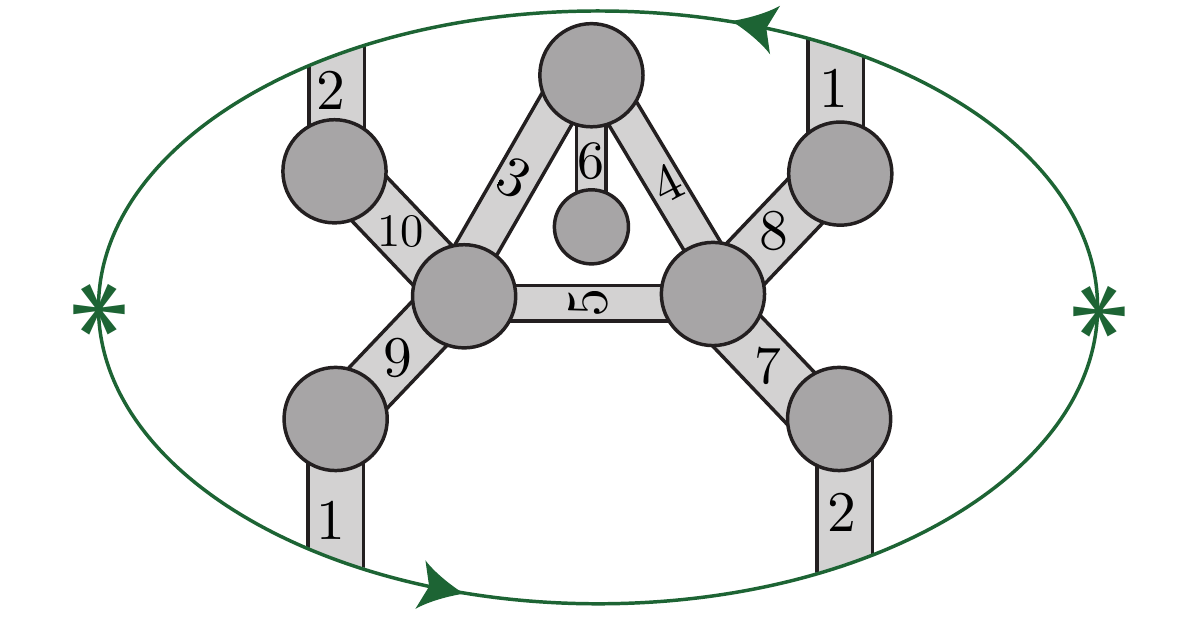}
\label{f.b1}
}
\subfigure[ The set $\mathcal{G}$ for  $A=\{1,2,3,4,5\}$.]{
\includegraphics[height=3cm]{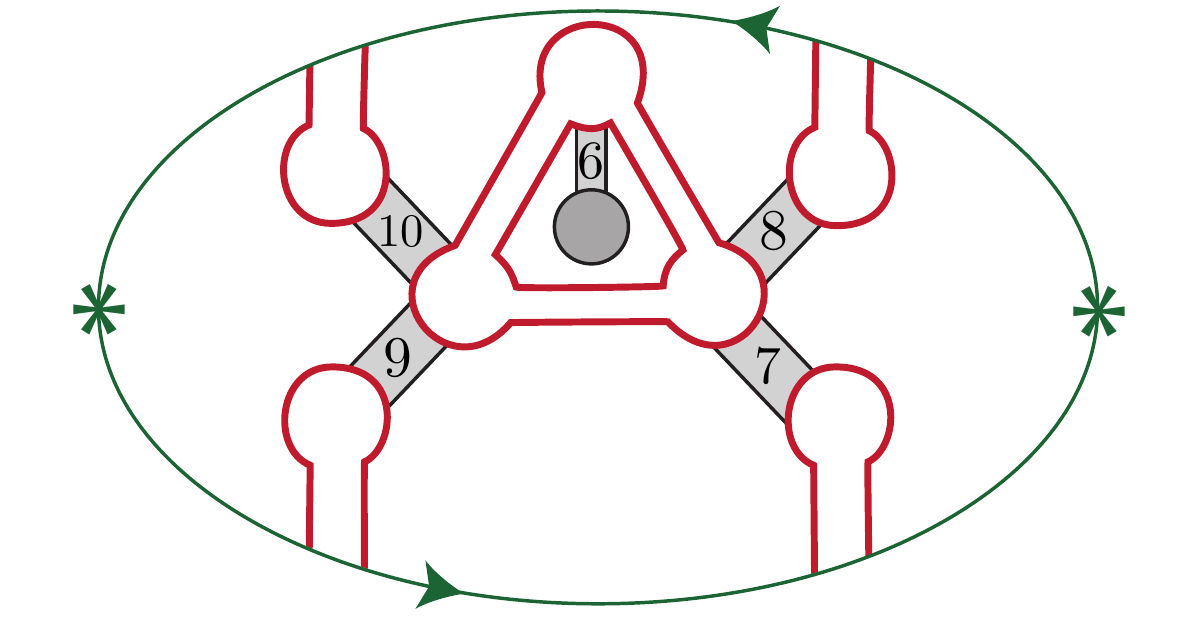}
\label{f.b2}
}
%\hspace{11mm}
%\subfigure[A tree associated with a sequence of $1$-sums.]{
%\includegraphics[height=4cm]{a6}
%\label{f.a3}
%}
\subfigure[The corresponding ribbon graphs.]{
\includegraphics[height=3.5cm]{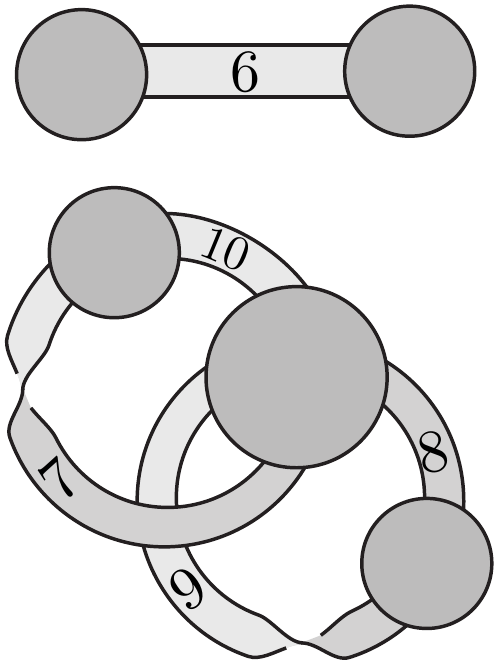}
\label{f.b3}
}
\caption{The construction of $(G^A)^A-A$ from the proof of Case~2 of Item~\ref{t2.2} of Theorem~\ref{t2}.}
\label{f.b}
\end{figure}

\begin{remark}
 Theorem~\ref{t2} tells us that \bis provide  characterizations of partial duals of plane and \RP graphs. However,  partial duals of higher genus ribbon graphs can not be characterized in terms of \bist. One can extend the concept of \pbis and \rpbis by saying  that $A$ defines a $\Sigma$-\bi of $G$ if it defines a \bi in which $\gamma(G|_A)+\gamma(G-A)=\gamma(\Sigma)$. It then follows from Theorem~\ref{t1} that $\gamma(G^A) =\gamma(\Sigma)$, {\em i.e.} if $A$ defines a $\Sigma$-\bi then  $G^A$ is a $\Sigma$ ribbon graph.
  The converse, however, does not hold: if $G^A$ is a $\Sigma$ ribbon graph, $A$ need not define a  $\Sigma$-\bi of $G$ (unless $\gamma(\Sigma)=0$ or $1$).
  For example, let $C$ be the plane $2$-cycle, and $D$ be the \RP $2$-cycle. If $e$ is an edge of $C$, then $C^{\{e\}}$ is a graph on a torus $T^2$, but  $\{e\}$ does not define a $T^2$-\bi as $\gamma(G|_A)=\gamma(G-A)=0$. Similarly,   if $e$ is an edge of $D$, then $D^{\{e\}}$ is a graph on a Klein bottle, $K$,  but  $\{e\}$ does not define a $K$-\bi as  $\gamma(G|_A)=\gamma(G-A)=0$. Counter examples for any $\gamma(G^A)>1$ can be obtained by joining $C$ or $D$ with toroidal or \RP ribbon graphs.
 
Extending Theorem~\ref{t2} to higher genus surfaces requires one to consider $n$-separations of ribbon graphs and is a work in progress.
\end{remark}

\subsection{Relating  \pbis and \rpbis }
Having established the importance of \pbis and \rpbis to partial duality in Theorem~\ref{t2}, we now go on to determine how all of the \pbis and \rpbis that a ribbon graph admits are related to one another.

\begin{definition}
For  $r\geq 1$, let  $G=H_1 \vee H_2 \vee \cdots \vee H_r$ be a ribbon graph. Suppose that  $A\subseteq E(G)$. Then we say that the set  $A\Delta E(H_i)$ is obtained from $A$ by  {\em toggling a  join-summand}. (See Figure~\ref{f.t}.) 

We say that two sets $A$ and $B$ are {\em related by toggling   join-summands} if there is a sequence of sets $A=A_1,A_2,\ldots , A_n=B$ such that each  $A_{i+1}$ is obtained from $A_i$ by    toggling a  join-summand.
\end{definition}

\begin{figure}
\includegraphics[height=25mm]{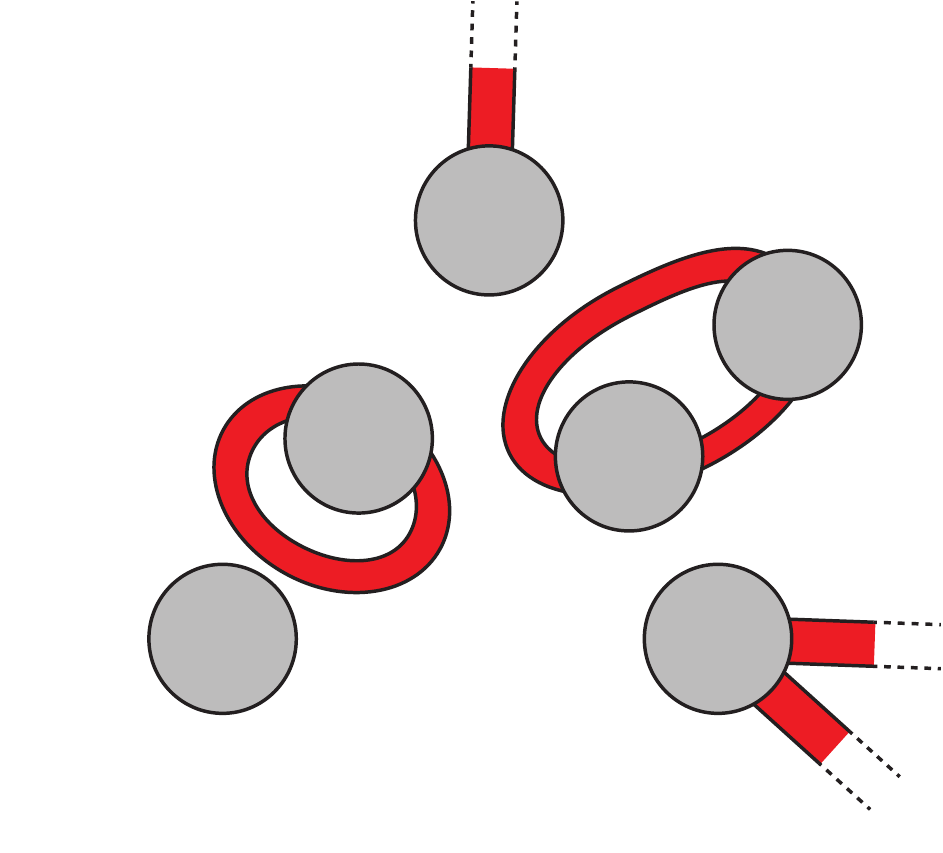}
\raisebox{12mm}{\rotatebox{180}{\includegraphics[width=1cm]{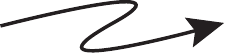}}}
\includegraphics[height=25mm]{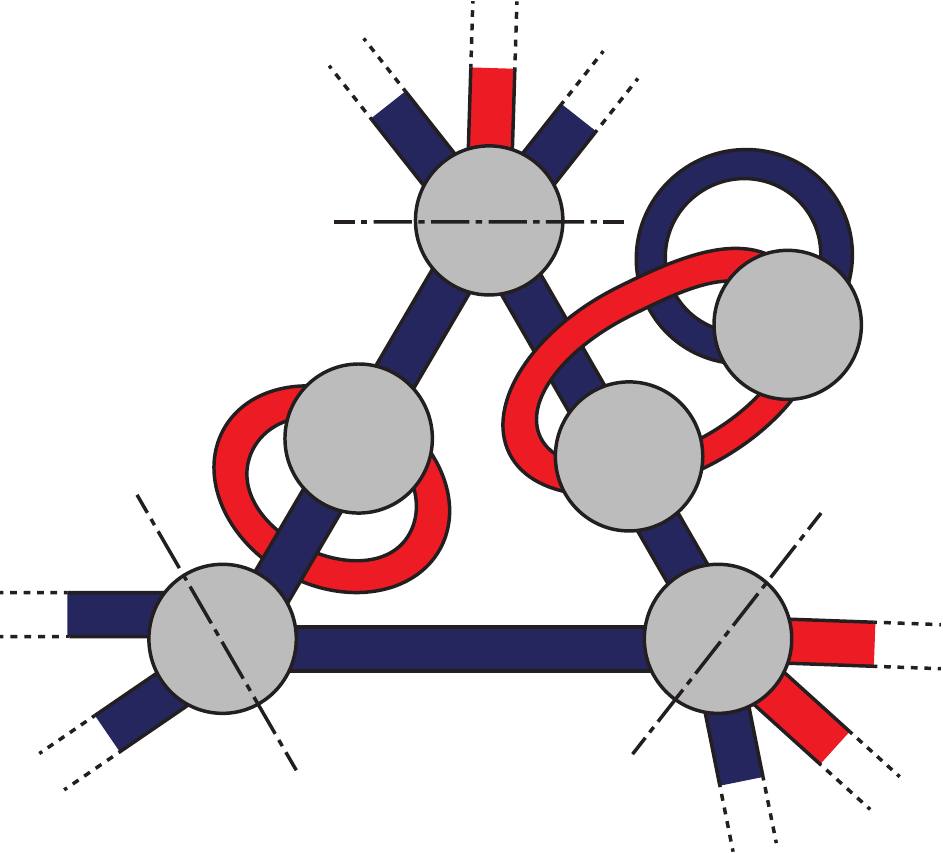} 
\raisebox{10mm}{\includegraphics[width=1cm]{doublearrow}}
\includegraphics[height=25mm]{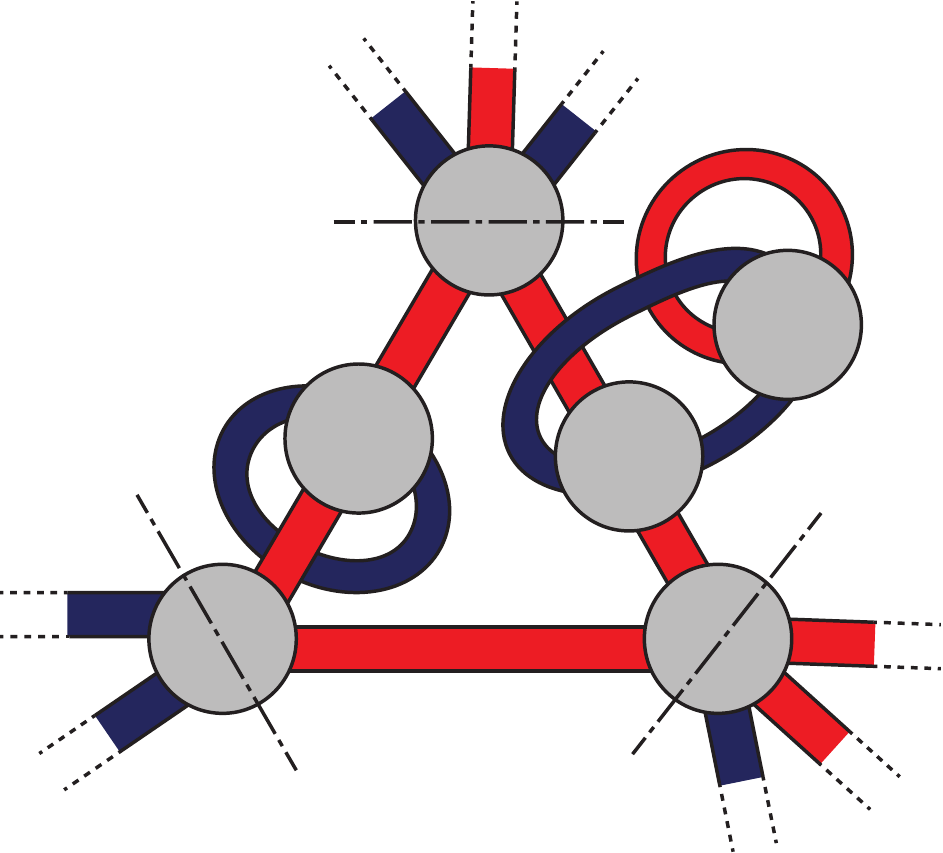}
\raisebox{10mm}{\includegraphics[width=1cm]{arrow}}
\includegraphics[height=25mm]{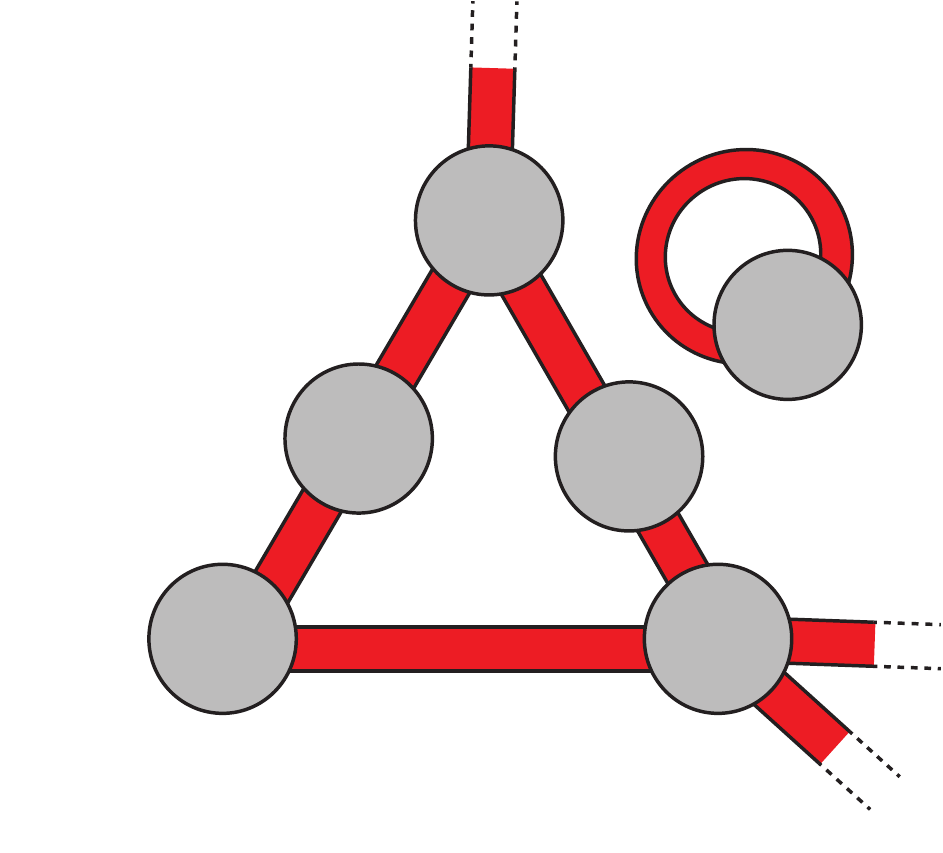}
\caption{Toggling a  join-summand and the  ribbon subgraphs induced by $A$ and $A\Delta E(H_i)$.
}
\label{f.t}
\end{figure}

The following theorem states that all of the \pbis or \rpbis that a ribbon graph admits are related to one another by toggling join-summands. The plane case in Theorem~\ref{t.tog} is from \cite{Mo5} and is included here for completeness.
\begin{theorem}\label{t.tog}
Suppose that  $G$ is  a connected ribbon graph and that $A, B\subseteq E(G)$ either both define   \pbit, or both define   \rpbist. Then $A$ and $B$ are related by toggling join-summands. 
\end{theorem}

\begin{example}
Let $G$ be the ribbon graph in Figure~\ref{f.s1}, and let  $B_1=E(G)\bs \{5,15\}$, $B_2=\{5\}$, and $B_3=\{15\}$. Further, for each $i$, let $H_i$ be the ribbon subgraph induced by $B_i$. Then 
$G=H_1\vee H_2\vee H_3$. The set $A=\{1,2,3,4, 9,14\}$ defines an \rpbi of $G$. We have
$A\Delta B_1=\{ 6,7,8,10,11,12,13 \}$,
$A\Delta B_2=\{ 1,2,3,4,5, 9,14 \}$, 
$A\Delta B_3=\{ 1,2,3,4, 9,14,15 \}$,
$A\Delta (B_1\cup B_2)=\{ 5,6,7,8,10,11,12,13 \}$,
$A\Delta (B_1\cup B_3)=\{ 6,7,8,10,11,12,13,15 \}$,
$A\Delta (B_2\cup B_3)=\{ 1,2,3,4,5, 9,14,15 \}$, and
$A\Delta (B_1\cup B_2\cup B_3)=\{ 5,6,7,8,10,11,12,13,15 \}$. Together with $A$, these are precisely the sets that define \rpbis in Item~\eqref{e.rpbis4} of Example~\ref{e.rpbis}.
\end{example}

To prove the theorem we use the following result about  \bis of a prime ribbon graph. A ribbon graph  is said to be {\em prime} if it can not be expressed as a join of (non-trivial) ribbon graphs.

\begin{lemma}\label{l.prime}
Let  $G$ be a connected prime ribbon graph.
\begin{enumerate}
\item Either $G$ does not admit a \pbit, or it admits exactly two \pbist, in which case the \pbis are defined by a set $A$ and its complement $A^c=E(G)\bs A$.

\item Either $G$ does not admit an \rpbit, or it admits exactly two \rpbist, and the \rpbis are defined by a set $A$ and its complement $A^c=E(G)\bs A$.
\end{enumerate}
\end{lemma}
\begin{proof}

We prove the second item first. Suppose that $G$ admits an \rpbit.  We will show that the assignment of any edge to either $A$ or $E(G)\bs A$  completely determines an \rpbi for $G$, and that the \rpbi is defined by $A$.

At each vertex $v_i$, partition the set of incident half-edges into blocks $\mathcal{O}_i, \mathcal{A}_{i,1}, \mathcal{A}_{i,2}, \mathcal{A}_{i,3}, \ldots$ according  the following rules: if two half-edges lie in a non-orientable cycle of $G$, place them in  the block $\mathcal{O}_i$; for the remaining half-edges, place them in the same block if and only if there is a path in $G$ between the two half-edges that does not pass through the vertex $v_i$, denoting the resulting blocks of the partition by $ \mathcal{A}_{i,1}, \mathcal{A}_{i,2},  \ldots$.  

We will now show that the blocks at $v_i$ give rise to exactly two possible assignments of the incident edges to the sets  $A$ and $A^c$ and that these assignments are complementary. 

If there is only one block $\mathcal{O}_i$ in the partition of the half-edges incident to $v_i$, then all of the edges incident to $v_i$ must appear in the unique \RP component of an \rpbi of $G$. Thus every edge incident to $v_i$ is either in $A$, or every edge is in $A^c$. 

If there is only one block $\mathcal{A}_{i,j}$  at $v_i$ then there is a path in $G$ that does not pass through $v_i$ between every  pair of half-edges in $\mathcal{A}_{i,j}$. It then follows that  there can not be a $1$-sum occurring at $v_i$ in any \rpbit.  Thus every edge incident to $v_i$ is either in $A$, or every edge is in $A^c$. 

Now suppose that the partition at $v_i$ contains more than one block. Denote these blocks by
$\mathcal{B}_1, \cdots , \mathcal{B}_d$, where each  $B_r$ is a block  $\mathcal{O}_i$ or $ \mathcal{A}_{i,j}$. Arbitrarily choose one of the two cyclic orderings of the half-edges incident to $v_i$.
We say that two blocks $\mathcal{B}_p$ and $\mathcal{B}_q$ {\em interlace} each other if there are half-edges $e, e'\in \mathcal{B}_p$ and $f, f'\in \mathcal{B}_q$ such that we meet the edges in the order $e,f,e',f'$ when travelling round the vertex $v_i$ with respect to the cyclic order. 

Observe that: 
\begin{enumerate}
\item every block $\mathcal{B}_p$ interlaces at least one other block $\mathcal{B}_q$.    (Otherwise $\mathcal{B}_p$ defines a join-summand and so $G$ is not prime.)
\item If $B$ is a set of blocks and $\bar{B}$ is the complementary set of blocks, then a block in $B$ interlaces a block in $\bar{B}$. (Otherwise $B$ and $\bar{B}$ define a join and so $G$ is not prime.)
\item In any \rpbit, all of the half-edges in a block must belong to $A$ or to $A^c$. (Since there is a path in $G$ that does not pass through $v_i$ between every  pair of half-edges in a block, it would follow that there  is a component of $G|_A$ and one of $G|_{A^c}$ that share  more than one vertex, and so $A$ would not define an \rpbit.)
\item In any \rpbit,  the half-edges in interlacing  blocks must belong to different sets $A$ or  $A^c$. (Otherwise, since $G$ is prime, the \rpbi would contain a non-plane orientable $1$-summand, or a non-orientable $1$-summand of genus greater than $1$.) 
\end{enumerate}
From these observations it follows that assigning any edge incident to $v_i$ to either $A$, or to $A^c$,  determines a unique assignment of every edge that is incident to $v_i$ to either $A$ or  $A^c$. Thus the $1$-sum at $v_i$ in the \rpbi is determined by the assignment of a single edge to  $A$ or to $A^c$. 

From the three cases above, since $G$ is connected, the assignment of any edge $e$ to $A$ will completely determine an \rpbit, and the assignment of  $e$ to $A^c$ will completely determine a second \rpbit, and the result follows. 

The second claim is from \cite{Mo5} and can be proven by replacing ``\rpbit'' with ``\pbit'' in the above argument, in which case each $\mathcal{O}_i$ is empty. \end{proof}

\begin{proof}[Proof of Theorem~\ref{t.tog}]
We prove the \RP case first. Suppose that $G$ admits an \rpbit.
Every ribbon graph $G$ admits a unique prime factorization $G=H_1 \vee \cdots \vee H_r$, for some $r\geq 1$ (see \cite{Mo5}).
Every  \rpbi of $G$ is uniquely determined by choosing an \rpbi of $H_i$ (if $H_i$ is non-orientable) or \pbi   of $H_i$ (if $H_i$ is orientable), for each $i$. Also   every choice of \rpbi or \pbi for the $H_i$'s gives rise to an \rpbi of $G$.
By Lemma~\ref{l.prime}, each $H_i$ admits either exactly two \rpbist, or exactly two \pbist,   and these are obtained  by toggling the assignment of the edges in $E(H_i)$ to  $A$ and to $E(G)\bs A$. Such a move replaces a set $A\subseteq E(G)$ with $A\Delta E(H_i)$. 
It follows that   $A$ and $B$ both define \rpbis of $G$ if and only if $A$ can be obtained from $B$ by toggling join-summands.

The first item in the theorem follows by replacing ``\RPt'' with ``plane'' in the above argument. 
\end{proof}

Note that although \pbis and \rpbis are related to one another by toggling join-summands, this is not the case for \bis in general. However, the results in this subsection  can easily be  extended to all \bis in which there is at most one non-plane component.

\section{partial duals of the same genus}\label{s5}
We saw in Theorem~\ref{t2} that partial duals of plane and \RP graphs are completely characterized by \pbis and \RPt-\bist. In this section we use this characterization to study partially dual plane graphs and partially dual \RP graphs.  We begin by introducing the notions of 
  \pjbis and   \rpjbist. These are special types of \bis that are based on joins, rather than $1$-sums. We show, in Theorem~\ref{t3}, that partially dual plane graphs, and partially dual \RP graphs are characterized by \pjbis and   \rpjbist, respectively. We then use this fact to find a local move on ribbon graphs that relates all partially dual plane and \RP graphs. Again, the results for plane graphs presented in this section are from \cite{Mo5} and are included here to illustrate the unified approach to partial duality for low genus ribbon graphs.

\subsection{Join-biseparations}
By way of motivation, suppose that we have a ribbon graph $G$  that can be written as a sequence of joins $H_1\vee \cdots \vee H_l$ in which every join occurs at the {\em same} vertex.  
It  follows that for each $I\subseteq \{1, \ldots , l\}$, the set $A= \bigcup_{i\in I} E(H_i)$ defines a \bi of $G$. Using the fact that joins preserve genus,  and using Theorem~\ref{t1}, we see that $\gamma(G)=\sum_{i=1}^l \gamma(H_i)=\gamma(G^A)$. That is, for this type of \bit, partial duality does not change genus. We extend this type of \bi  by dropping the requirement that the joins all occur at the same vertex, to define a \jbit.
\begin{definition}\label{d.jbi}
Let $G=(V,E)$ be a ribbon graph and $A\subseteq E(G)$.    We say that $A$ defines a {\em \jbi} of $G$ if we can write $G=H_1\vee \cdots \vee H_l$, for some $l\geq 1$, where $A= \bigcup_{i\in I} E(H_i)$ for some $I\subseteq \{1, \ldots , l\}$.

If, in addition, each $H_i$ is plane we say that $A$ defines a {\em \pjbit}; and if exactly one $H_i$ is \RP and all of the others are plane, then we say that $A$ defines a {\em \rpjbit}.
\end{definition}

\begin{example}\label{e.jb} Two examples of \jbis are given below.
\begin{enumerate}
\item \label{e.jb1} Let $G$ be the ribbon graph in Figure~\ref{f.s1}. Then $\emptyset$, $\{5\}$, $\{15\}$, $\{5,15\}$, $E(G)\bs\{5,15\}$, $E(G)\bs\{5\}$,  $E(G)\bs\{15\}$ and $E(G)$ are all of the sets that define \jbis of $G$.

\item \label{e.jb2}
Let $A_1=\{1\}$, 
$A_2=\{2\}$, 
$A_3=\{3\}$,  and
$A_4=\{4,5\}$.  
Then $A$ defines an \rpjbi of the ribbon graph $G$ (or of $H$) shown in Figure~\ref{f.sb} if and only if $A= \bigcup_{i\in I} A_i$, for some $I\subseteq \{1, 2,3, 4\}$.
\end{enumerate}
\end{example}

\begin{figure}
\centering
\subfigure[A ribbon graph $G$.]{
\includegraphics[scale=.5]{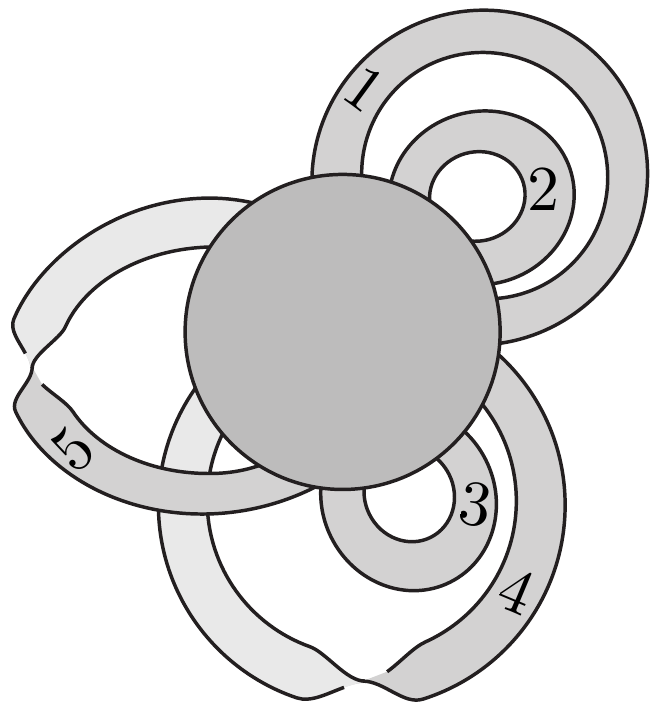}
\label{f.sb1}
}
\hspace{21mm}
\subfigure[A ribbon graph $H$.]{
\includegraphics[scale=0.5]{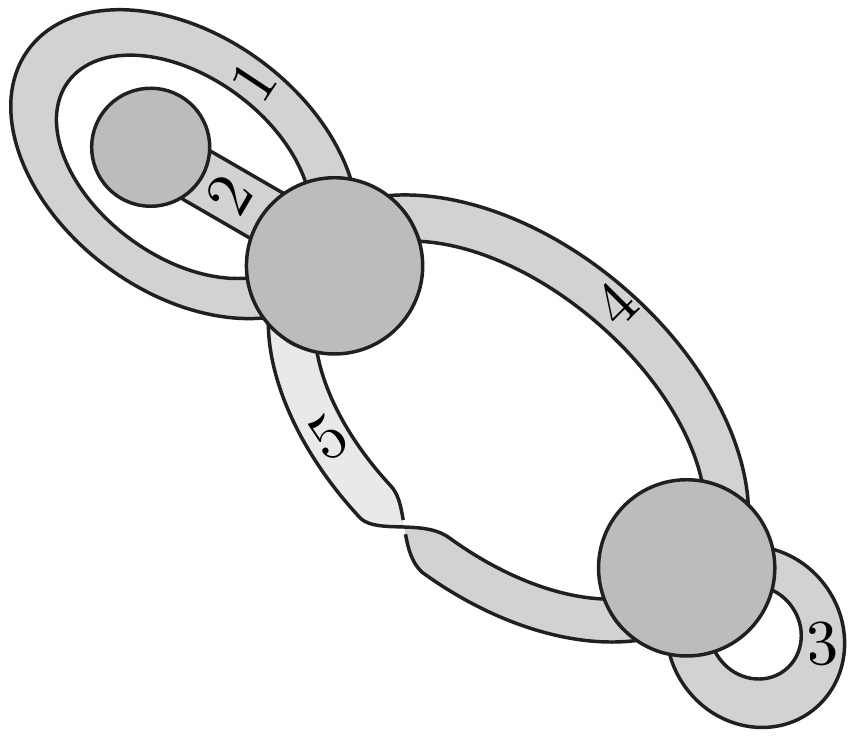}
\label{f.sb2}
}
\caption{Partially dual \RP ribbon graphs.}
\label{f.sb}
\end{figure}

For Definition~\ref{d.jbi}, we emphasize  that the joins need not occur at distinct vertices, that the join operation is not associative, and that the join summands need not be prime.

The following proposition states that \jbis are indeed \bist. In fact, we will see in Lemma~\ref{l2} that for plane graphs, \pjbis and \bis are equivalent; and that for \RP graphs, \rpjbis and \rpbis are equivalent. 
\begin{proposition}\label{p4}
Let $G$ be a connected ribbon graph and $A\subseteq E(G)$.  
\begin{enumerate}
\item\label{p4.1} If $A$ defines a \jbi of $G$, then it also defines a \bi of $G$. 
\item\label{p4.2}  If $A$ defines a \pjbi of $G$, then it also defines a \pbi of $G$. 
\item\label{p4.3}  If $A$ defines an \rpjbi of $G$, then it also defines an \rpbi of $G$. 
\end{enumerate}
\end{proposition}
\begin{proof}
Suppose $G=H_1\vee \cdots \vee H_l$, where $l\geq 1$, and $A= \bigcup_{i\in I} H_i$, for some $I\subseteq \{1, \ldots , l\}$. 

For Item~\ref{p4.1}, suppose, as we read from left to right, the joins in  $H_1\vee \cdots \vee H_l$ occur at the vertices $v_{\iota_1}, v_{\iota_2}, \ldots, v_{\iota_{l-1}}$ of $G$ in that order. The $v_{\iota_j}$ need not all be distinct. Reading through this sequence of vertices from the left, let $v_1, \ldots, v_p$ be the sequence of distinct vertices obtained by taking the first occurrence of each vertex in the sequence. Let $K_1, \ldots , K_{p+1}$ be the components of $G|_A$ and $G|_{A^c}$. Then 
$G$ can then be written as $G=K_0 \oplus K_1 \oplus \cdots \oplus K_{p+1}$, where the $1$-sums occur at $v_1, \ldots, v_p$, in that order, and involve a component of $G|_A$ and of $G|_{A^c}$. Thus $A$ defines a \bi of $G$.

 Item~\ref{p4.2} follows since, if all of the $H_i$ are plane, then so are the components of  $G|_A= \cup_{i\in I} H_i$ and $G|_{A^c}= \cup_{i\not\in I} H_i$.
 
 Similarly,  Item~\ref{p4.3} follows since if exactly one  $H_i$ is \RP with all of the others plane, then one component of  $G|_A= \cup_{i\in I} H_i$ or of $G|_{A^c}= \cup_{i\not\in I} H_i$ is \RP and all of the others are plane.
\end{proof}

\subsection{Plane and \RP partial duals}
Suppose that $A$ defines an \rpjbi of $G$. Then $G$ is an \RP ribbon graph. By Proposition~\ref{p4}, every \rpjbi is an \rpbit, and it follows from Theorem~\ref{t2} that $G^A$ is also an \RP ribbon graph. Thus, if  $A$ defines an \rpjbi of $G$, then  $G$ and $G^A$ are partially dual \RP ribbon graphs. A similar statement holds for plane graphs. In the following theorem we show that the converses of these statements also hold, giving a characterization of plane and \RP partial duals.  
\begin{theorem}\label{t3}
Let $G$ be a connected ribbon graph, and  $A\subseteq E(G)$. Then
\begin{enumerate}
\item $G$ and $G^A$ are both plane if and only if $A$ defines a \pjbi of $G$;
\item $G$ and $G^A$ are both \RP if and only if $A$ defines an \rpjbi of $G$.
\end{enumerate}
\end{theorem}

To prove the theorem we  need the following  lemma, which provides a converse to Proposition~\ref{p4}.
\begin{lemma}\label{l2}
Let $G$ be a  connected  ribbon graph.
\begin{enumerate}
\item \label{l2.1}  If $G$ is plane and  $A\subseteq E(G)$ defines a \pbi of $G$ then $A$ also defines a \pjbi of $G$. 

\item \label{l2.2} If  $G$ is  \RP and   $A\subseteq E(G)$ defines an \rpbi  of $G$, then $A$ also defines an \rpjbi of $G$.
\end{enumerate}
\end{lemma}

\begin{proof}
Item~\ref{l2.1} is from \cite{Mo5}.

For Item~\ref{l2.2}, if $A=E(G)$ or $A=\emptyset$ the result is trivial, so assume that this is not the case.  Then we can write
\begin{equation}\label{e.1v1}G=   K \oplus H_1  \oplus \cdots\oplus H_l,\end{equation}
where $K$ is an \RP ribbon graph and the $H_i$ are plane. (We are using Proposition~\ref{p.start}  to ensure that $K$ is the first $1$-summand.)
As every ribbon graph  admits a  prime factorization (\cite{Mo5}), we can write $K$ as a sequence of joins 
\begin{equation}\label{e.1v2}K=  J\vee J_1\vee \cdots \vee J_p,   \end{equation}
where $J$ is \RPt; the $J_i$ are plane (as $K$ is \RPt); and $J$ is prime, {\em i.e.}, it can not be written as a join of ribbon subgraphs.  Substituting \eqref{e.1v2} into \eqref{e.1v1} gives 
\begin{equation}\label{e.1v3}G=  J\vee J_1\vee \cdots \vee J_p\oplus H_1  \oplus \cdots\oplus H_l,   \end{equation}
where $J$ is \RP and the other terms are plane.

Now let $I_1, \ldots , I_m$ be the set of ribbon graphs obtained by taking prime factorizations of the set of components of $G-E(J)$, and then joining together all of the prime join-summands  that do not occur at a vertex in $V(J)$.
 Then (using the fact that $J$ is a $1$-summand in Equation~\eqref{e.1v3}) we can write
\begin{equation}\label{e.1v4}G=  J\oplus I_1  \oplus \cdots\oplus I_m.  \end{equation}
Each $I_j$ is plane (since if  $I_j$ was orientable and non-plane,  $G$ would not be \RPt; and if any $I_j$ was non-orientable it would follow that one of the $J_i$ or $H_i$ in \eqref{e.1v3} is non-orientable). Also, every $1$-sum in Equation~\eqref{e.1v4} occurs at a (not necessarily distinct) vertex of $J$.

We will now show that each $1$-sum in  Equation~\eqref{e.1v4} is in fact a join. 
Suppose that one of the $1$-sums is not a join. Suppose also that the $1$-sum occurs at a vertex $v$ and involves $I_j$ and (necessarily) $J$. Then $I_j$ must have two half-edges $e$ and $e'$ that are interlaced by  half-edges $f$ and $f'$ of $J$ when reading around $v$ with respect to either cyclic order (so the edges are met in the cyclic order $e\, f\, e'\, f'$). Now, $e$ and $e'$ must belong to a cycle $C$ in $I_j$ (otherwise the $I_j$ have not been constructed properly as it could be expressed as a join at a vertex of $J$ in $G-E(J)$). When $G$ is cellularly embedded in \RPt, the cycle $C$ can not be contractible (otherwise $J$ is not prime). Also, the cycle $C$ can not be non-contractible as (otherwise $I_j$ is not plane). This gives a contradiction. It follows that every $1$-sum in \eqref{e.1v4} must be a join.     
Thus we can write 
 \begin{equation}\label{e.1v5}G=  J\vee I_1  \vee \cdots\vee I_m,  \end{equation}
where each $I_j$ is plane and $J$ is \RPt.

Now, since each $I_j$ is a plane ribbon subgraph of $G$, the set $A\cap E(I_j)$ defines a \pbi of $I_j$. Then, by Item~\ref{l2.1} of the lemma, $A\cap E(I_j)$  defines a \pjbi of $I_j$. Thus we can write 
\[  I_j  = I_{j,1}  \vee \cdots\vee I_{j,p_j} ,  \]
where $A\cap E(I_j) = \bigcup_{k\in K} E(I_{j,k})$, for some index $K$.
Using this and Equation~\eqref{e.1v5}  we then have that 
\[   G=  J \vee I_{1,1}  \vee \cdots\vee I_{m,p_m},   \]
where $A=  \bigcup_{(l,k)\in I} E(I_{l,k})$ or $A= E(J)  \bigcup_{(l,k)\in I} E(I_{l,k})$ for a suitable index $I$, and thus $A$ defines an \rpjbi of $G$ as required.
\end{proof}

%\begin{figure}
%\begin{tabular}{ccc}
%\includegraphics[height=25mm]{dk1} & \hspace{1cm} &
%\includegraphics[height=25mm]{dk2} \\
%A ribbon graph embedded in \RPt. && The ribbon graphs $K$, $K_1$ and $K_2$.
%\end{tabular}
%\caption{An example of a construction  used in the proof of Lemma~\ref{l2}.}
%\label{f.dk}
%\end{figure}

\begin{proof}[Proof of Theorem~\ref{t3}.]
We will prove the second item first.
If  $G^A$ is an \RP ribbon graph then, by Theorem~\ref{t2}, $A$ defines an \rpbi of $G$. As $G$ is \RPt, it follows from Lemma~\ref{l2} that  $A$  defines an \rpjbi of $G$. 

Conversely,  if $A$ defines an \rpjbi of $G$, then by Proposition~\ref{p4} it also defines an \rpbi of $G$, and so $G^A$ is an \RP ribbon graph by Theorem~\ref{t2}. Also, as joins preserve genus and orientability, $G$ is also an \RP ribbon graph.

The first item of the theorem follows by replacing ``\RP'' with ``plane'' in the above argument.
\end{proof}

\begin{remark}
The characterization of partially dual plane and \RP graphs given in Theorem~\ref{t3} does not extend to higher genus ribbon graphs. That is, if $G$ is a ribbon graph and $A\subseteq E(G)$ such that $\ga(G)=\ga(G^A)\geq 2$, it does not follow that $A$ defines a \jbi of $G$.  In fact, even if $G$ is a ribbon graph and $A\subseteq E(G)$ defines a \bi of $G$ such that $\gamma(G)=\ga(G^A)\geq 2$, it still does not follow that $A$ defines a \jbi of $G$. For example, let $G$ be the orientable ribbon graph with one vertex and  three edges $a,b,c$ that are met in the cyclic order $a\,b\,c\,a\,c\,b$, and let $A=\{a\}$. Then $A$ defines a \bi of $G$ with $\gamma(G)=\gamma(G^A)=2$, but $A$ does not define a \jbi of $G$. A non-orientable example with  $\gamma=2$ can be obtained  by adding a half-twist to the edges $b$ and $c$ in the above example. Higher genus examples can be obtained by joining toroidal or \RP ribbon graphs to these two examples.

It is also worth noting that while it follows from  Lemma~\ref{l2} that every \bi of a plane graph is a \jbit, it is not true, however, that every \bi of an \RP graph is an \rpbit. For example let $G$ be the \RP ribbon graph  with one vertex and two edges $a,b$ in the cyclic order $a\,b\,a\,b$. Then $A=\{a\}$ defines a \bi that is not an \rpbit.
\end{remark}

\subsection{Relating plane and \RP partial duals}
We now introduce a simple local move on ribbon graphs, called dualling a  join-summand. We go on to  show that this move  relates all  partially dual \RP and plane ribbon graphs.

\begin{definition}
Let $G=P\vee Q$ be a ribbon graph. We  say that the ribbon graph   $G^{E(Q) } =  P\vee Q^{ E(Q)}=P \vee Q^*$ is obtained from $G$ by a  {\em dual-of-a-join-summand move}. We say that two ribbon graphs are related by {\em dualling  join-summands} if there is a sequence of  dual-of-a-join-summand moves taking one to the other, or if they are geometric duals. 
\end{definition}

\begin{theorem}\label{t4}
Let $G$ and $H$ be connected ribbon graphs.
\begin{enumerate}
\item If $G$ and $H$ are both plane, then they are partial duals of each other if and only if they are related by dualling join-summands.
\item  If $G$ and $H$ are both \RPt, then they are partial duals of each other  if and only if they are related by dualling join-summands.
\end{enumerate}
\end{theorem}

\begin{example}
 The \RP ribbon graphs shown in Figure~\ref{f.sb} are partial duals: $H=G^{\{2,4,5\}}$.  It is readily checked that  $H$ can be obtained from $G$ by dualling the  join-summands determined by the following edge sets in the given order: $\{ 3,4,5 \}$ then $\{3\}$ then $\{2\}$. This sequence is not unique.
\end{example}

We will use the following lemma in the proof of Theorem~\ref{t4}.
\begin{lemma}\label{l5}
~
\begin{enumerate}
\item \label{l5.1} $(P\vee Q)^A = P^{E(P)\cap A} \vee Q^{E(Q)\cap A}$.
\item\label{l5.2}  If $G= H_1\vee \cdots \vee H_l$, then, for each $i$, $G$ and $G^{E(H_i)}$ are related by dualling join summands.
\end{enumerate}
\end{lemma}
\begin{proof}
Item~\ref{l5.1} is from \cite{Mo5}.

For Item~\ref{l5.2}, let 
\begin{equation}\label{l5.e1}G= H_1\vee \cdots \vee H_l.\end{equation}
In $G- (E(H_1)\cup\cdots \cup E(H_{i-1}))$, let 
 $K=H_i\vee H_{i_1} \vee \cdots \vee H_{i_p}$ be the component  that contains $H_i$, and let $J_1, \ldots, J_q$ denote the other non-trivial components.
We can reorder the joins in Equation~\eqref{l5.e1} to get
\begin{equation*}
G= H_1\vee \cdots \vee H_{i-1} \vee J_1 \vee \cdots \vee J_q \vee K 
= H_1\vee \cdots \vee H_{i-1} \vee J_1 \vee \cdots \vee J_q \vee H_i\vee H_{i_1} \vee \cdots \vee H_{i_p}.
\end{equation*}
Then, by Item~\ref{l5.1} of the Lemma and Proposition~\ref{p.pd2}, we have
\begin{multline*}
G^{E(H_i)} = (G^{E(K)})^{E(K)\bs E(H_i)}
= (H_1\vee \cdots \vee J_q \vee K^*)^{E(K)\bs E(H_i)} 
=  (H_1\vee \cdots \vee J_q \vee H_i^*\vee H_{i_1}^* \vee \cdots \vee H_{i_p}^*)^{E(K)\bs E(H_i)} \\
= H_1\vee \cdots \vee J_q \vee H_i^*\vee (H_{i_1}^* \vee \cdots \vee H_{i_p}^*)^{E(K)\bs E(H_i)}
= H_1\vee \cdots \vee J_q \vee H_i^*\vee H_{i_1} \vee \cdots \vee H_{i_p})\\
= H_1\vee \cdots  \vee H_i^* \vee \cdots \vee H_l.
\end{multline*}
Upon observing that the above sequence is just the application of two dual-of-a-join-summand moves, the result follows.
\end{proof}

\begin{proof}[Proof of Theorem~\ref{t4}] 
It is clear that if $G$ and $H$ are related by dualling  join-summands then they are partial duals.

Conversely, suppose that $H=G^A$ for some $A\subseteq E(G)$.
Since  $G$ and $G^A$ are both plane or both  \RP ribbon graphs, it follows from Theorem~\ref{t3} that
$G= H_1\vee \cdots \vee H_l $, where $l\geq 1$,   $A=\bigcup_{i\in I} E(H_i) $, and $I=\{\iota_1, \ldots , \iota_p\}$.  We can then write
\[ G^A=  (\cdots((H_1\vee \cdots \vee H_l)^{H_{\iota_1}})^{H_{\iota_2}}) \cdots )^{H_{\iota_p}} , \]
which by Item~\ref{l5.2} of Lemma~\ref{l5} can be obtained from $G$ by a sequence of  dual-of-a-join-summand moves.
\end{proof}

\section*{Acknowledgements}
I would like to thank Lowell Abrams for stimulating conversations.
%, and the anonymous referees for their careful reading and helpful comments.

\end{document}